%% file: SingSSF_Push_invariant_v4_0.tex
\documentclass[12pt]{amsart} 
\usepackage{amsfonts,amssymb,amsmath,euscript}
\input Macros
\input MyMakeLit

\ndef{\mbCz}{\mbC^{(z)}}
\ndef{\mbCr}{\mbC^{(r)}}

\ndef{\yy}{y}

\newcommand{\LambHF}[1]{\Lambda(#1;F)}

\ndef{\ells}{{\ell_2}}
\ndef{\hlambda}{{\mathfrak h_\lambda}}
\ndef{\hlambdao}{{\mathfrak h_\lambda^{(0)}}}
\ndef{\hlambdar}{{\mathfrak h_\lambda^{(r)}}}

\begin{document}
\title[Singular SSF]{Singular spectral shift \\ and Pushnitski $\mu$-invariant}
\author{\Azamov}
\address{School of Computer Science, Engineering and Mathematics
   \\ Flinders University
   \\ Bedford Park, 5042, SA Australia.}
\email{azam0001@csem.flinders.edu.au}

\begin{abstract} In this paper it is shown that in case of trace class perturbations
the singular part of Pushnitski $\mu$-invariant does not depend on the angle variable.
This gives an alternative proof of integer-valuedness of the singular part of the spectral shift
function. As a consequence, the Birman-Krein formula for trace class perturbations follows.
\end{abstract}
\maketitle

\bigskip

\tableofcontents

%
%

\section{Introduction}

Let $H_0$ be a self-adjoint operator, let $V$ be a trace-class self-adjoint operator
and let $H_r = H_0+rV,$ where $r \in \mbR.$
The Kato-Rosenblum theorem asserts that in this case there exist M\"oller's
wave operators
$$
  W_\pm(H_r,H_0) := s-\lim_{t \to \pm \infty} e^{itH_r}e^{-itH_0}P^{(a)}(H_0),
$$
where the limit is taken in the strong operator topology and where $P^{(a)}(H_0)$ is
the projection onto the absolutely continuous subspace $\hilba(H_0)$ of $H_0.$
The scattering matrix is defined by formula
$$
  \mathbf S (H_r,H_0) = W_+^*(H_r,H_0)W_-(H_r,H_0).
$$
The scattering matrix $\mathbf S (H_r,H_0)$ is a unitary operator from $\hilba(H_0)$ to $\hilba(H_0)$
which commutes with $H_0.$ It follows that in the direct integral decomposition
$$
  \euF \colon \hilba(H_0) \to \int^\oplus_{\hat\sigma} \hlambda\,d\lambda
$$
of the Hilbert space $\hilba(H_0),$ which diagonalizes $H_0:$
$$
  \euF(H_0 f)(\lambda) = \lambda \euF(f)(\lambda), \ \ \text{a.e.} \ \lambda,
$$
the operator $\mathbf S (H_r,H_0)$ has the form
$$
  \mathbf S (H_r,H_0) = \int^\oplus_{\hat\sigma} S(\lambda; H_r,H_0)\,d\lambda,
$$
where $S(\lambda; H_r,H_0)$ is a unitary matrix for a.e. $\lambda,$ called the scattering matrix,
and where $\hat \sigma$ is a core of the spectrum of $H_0.$

There is a formula for the scattering matrix, established rigorously by Birman and \`Entina in \cite{BE} (see also \cite{Ya})
(which was well-known to physicists well before \cite{BE}),
$$
  S(\lambda; H_r,H_0) = 1_\lambda -2\pi i Z(\lambda; G)(1_\lambda + JT_0(\lambda+i0))^{-1}Z^*(\lambda; G) \ \ \text{a.e.} \ \lambda \in \mbR,
$$
called the stationary formula, where the perturbation $V$ is factorized as $G^*JG$ with $G\colon \hilb \to \clK$ a Hilbert-Schmidt operator
from $\hilb$ to an auxiliary Hilbert space $\clK,$ $J\colon \clK \to \clK$ is a bounded self-adjoint operator, $T_0(\lambda+i0)$
is the Hilbert-Schmidt limit of $G^* (H_0-(\lambda+iy))^{-1}G$ as $y \to 0^+$ and $Z(\lambda; G)f = (\euF G^*f)(\lambda).$

The scattering matrix $S(\lambda; H_r,H_0)$ is a unitary matrix such that the matrix $S(\lambda; H_r,H_0) - 1_\lambda$
is trace class in $\hlambda.$ So, the spectrum of $S(\lambda; H_r,H_0)$ consist of the eigenvalues of $S(\lambda; H_r,H_0)$
(called scattering phases) with only one possible accumulation point $1.$ In \cite{Pu01FA}, \Pushnitski\ introduced the so-called $\mu$-invariant
$\mu(\theta,\lambda)$ which calculates the spectral flow of eigenvalues of $S(\lambda; H_r,H_0)$
through the point $e^{i\theta}.$ To be able to calculate the spectral flow it is necessary to find
a certain way to send the eigenvalues of the scattering matrix to $1.$ Roughly speaking, the $\mu$-invariant
calculates the spectral flow when the imaginary part $y$ of the spectral parameter $\lambda+iy$ is sent from $+\infty$ to $0.$

In \cite{Az2}, for a class of Schr\"odinger operators, another way was discovered to calculate the spectral flow of the scattering phases: by sending the coupling constant
$r$ from $1$ to $0.$ It was also observed in \cite{Az2} that these two ways may not necessarily give the same answer.
The second way of calculating the spectral flow of scattering phases was called the absolutely continuous part of the $\mu$-invariant
and denoted by $\mua(\theta,\lambda).$ One of the observations made in \cite{Az2} was that the difference between these two ways
to calculate the spectral flow of scattering phases was equal to (the density of) the so-called singular part of the spectral shift function
$\xis(\lambda; H_r,H_0).$

In the definition of $\mua(\theta;\lambda)$ for arbitrary self-adjoint operators, there is one significant technical difficulty: one has to make sure
that for a.e. $\lambda$ the scattering matrix $S(\lambda; H_r,H_0)$ exists for all $r \in [0,1].$
This difficulty was overcome in \cite{Az3v4}, where it was shown that (provided there is a fixed Hilbert-Schmidt
operator $F$ with trivial kernel in $\hilb,$ called a frame operator) $S(\lambda; H_r,H_0)$ exists for all $r$
except a discrete set, and that $S(\lambda; H_r,H_0)$ is a meromorphic function of the coupling constant $r$ which admits
analytic continuation to all $\mbR.$

In \cite{Pu01FA} the formula
$$
  \xi(\lambda; H_1, H_0) = - \frac 1{2\pi} \int_0^{2\pi}\mu(\theta, \lambda; H_1,H_0)\,d\theta
$$
was proved for a.e. $\lambda \in \mbR.$ In this paper I prove analogue of this formula for the absolutely continuous part of $\mu:$
$$
  \xia(\lambda; H_1, H_0) = - \frac 1{2\pi} \int_0^{2\pi}\mua(\theta, \lambda; H_1,H_0)\,d\theta.
$$
I also show that the difference
$$
  \mus(\theta, \lambda; H_1,H_0) := \mu(\theta, \lambda; H_1,H_0) - \mua(\theta, \lambda; H_1,H_0)
$$
does not depend on the angle variable $\theta.$ This implies that
$$
  \xis(\lambda) = - \mus(\theta, \lambda; H_1,H_0),
$$
and, as a consequence, that $\xis(\lambda)$ is a.e. integer-valued. In \cite{Az3v4} it was shown that integer-valuedness
of $\xis$ is equivalent to the Birman-Krein formula. As such, this paper gives a new proof of the Birman-Krein formula.

To make sure that the two ways to calculate spectral flow of scattering phases are indeed different,
one has to prove existence of non-trivial singular spectral shift functions $\xis.$ This has been done in \cite{Az6}.

The first five sections are devoted essentially to an exposition of Pushnitski $\mu$-invariant.
The approach of this paper to the notion of $\mu$-invariant is different from that of \cite{Pu01FA}.
In section 7 the $\mu$-invariant is used to give a new proof of integer-valuedness of $\xis.$

\section{The metric space $\euS_1(\mbT)$}
In this section we shall study the space $\euS_1(\mbT)$ of spectra of unitary operators $U,$ such that $U-1$ is
trace class. This space plays a special role for the notion of $\mu$-invariant. The space $\euS_1(\mbT)$ was introduced and studied in \cite{Pu01FA}
(where it was denoted by $X_1$), but in a different way. The treatment of $\euS_1(\mbT)$ given in this and subsequent sections is more suitable for what follows
in further sections.

Many proofs in this section are probably standard and/or quite simple.
I do not claim originality for them and include them for completeness.

The space $\euS_1(\mbT)$ is a metric space, and as such it is desirable to study it first from purely topological
point of view; this will clarify its applications to the study of $\mu$-invariant.

\newcommand{\mult}{\mathrm{mult}}
\subsection{Rigged sets and their enumerations}

A \emph{rigged subset} $S$ of a set $X$ is a set of elements of $X$ with assigned number $n=0,1,2,\ldots$ or infinity $\infty:$
$\set{(s,n)\colon s \in X, n =1,2,\ldots}.$ The number $n$ of the pair $(s,n)$ should be understood as \emph{multiplicity} of the element $s.$
If $(s,n) \in S$ with $n=0,$ the corresponding element $s$ is deemed not to belong to the rigged set $S,$
and we write $s \notin S;$ while $s \in S$ means that $n\geq 1.$
So, any element of $X$ can be considered as an element of $S$ with appropriate multiplicity.
The number $n$ will be called multiplicity of the element $s$ and will be denoted by $\mult(s) = \mult(s;S).$

Given two rigged subsets $S_1$ and $S_2$ of $X,$ we write $S_1 \leq S_2,$ if, for any $x \in X,$
$$
  \mult(x; S_1) \leq \mult(x; S_2).
$$
One can naturally define the sum $S_1+S_2$ and, in case of $S_2 \leq S_1,$ the difference $S_1-S_2$  of two rigged sets $S_1$ and $S_2$
by formulas
\begin{gather*}
  \mult(x; S_1 \pm S_2) = \mult(x; S_1) \pm \mult(x; S_2). 
\end{gather*}
It is also possible to define union and intersection of two rigged sets, but we don't need them.

A set $X$ can be treated as a rigged set all elements of which has multiplicity $1.$
To distinguish rigged sets from usual sets we use notation $\set{\ldots}^*$ for rigged sets.
If $S$ is a rigged set and if $T$ is a usual set, the inclusion $S \subset T$ will mean
that any element of $S$ (with non-zero multiplicity) belongs to $T.$ Also, by $S \cap T$ we denote the rigged
set
$$
  S \cap T := \set{x \in S \colon x \in T}^*,
$$
that is, $\mult(x; S \cap T) = \mult(x; S),$ if $x \in T;$ and $\mult(x; S \cap T) = 0,$ if $x \notin T.$

We define support $\supp(S)$ of a rigged subset $S$ of a (usual) set $X$ as a usual set by the formula
$$
  \supp(S) := \set{x \in X \colon \mult(x; S) > 0}.
$$
A rigged set is countable, if its support is countable. We consider only countable rigged sets.

By an enumeration of a countable rigged set $S$ we mean any sequence (or a double sequence)
of elements of~$S,$ in which every element~$s$ of~$S$ appears exactly $\mult(s)$ times.

Rank $\rank(S)$ of a rigged set $S$ is the sum of multiplicities of all its elements.

If $X$ is a topological space and $S$ is a rigged subset of $X,$ then, by definition, $x_0 \in X$ is an accumulation point of $S,$
if for any neighbourhood $U$ of $x_0$ the rank of $S \cap U$ is infinite. In particular, a point of infinite multiplicity
is an accumulation point.


\subsection{Definition of $\euS_1(\mbT)$}
Let $\mbX$ be a metric space with a fixed point $x_0 \in \mbX.$ By $\euS_\infty(\mbX) = \euS_\infty(\mbX, x_0)$
we denote the set of all rigged subsets of $\mbX$ with only one accumulation point $x_0 \in \mbT.$
In fact, we are interested in only two cases: $(\mbX, x_0) = (\mbR, 0)$ and, most of all, $(\mbX, x_0) = (\mbT, 1).$

The element $x_0$ itself is always supposed to have multiplicity $\infty:$ $\mult(x_0) = \infty.$
It follows that the support of any (rigged) set from $\euS_\infty(\mbX)$
can have only one accumulation point $x_0.$

Since $x_0$ has infinite multiplicity in any rigged set $S$ from $\euS_\infty(\mbX),$ we agree to say that the rank of $S \in \euS_\infty(\mbX)$
is the rank of the rigged set $S \setminus \set{x_0}.$ If the rank of $S$ is finite, we say that $S$ is finite or of finite rank.

If $S, T \in \euS_\infty(\mbX),$ then we let
\begin{equation} \label{F: d(S,T) def of}
  d(S,T) = \inf \sum\limits_{j=1}^\infty \abs{s_j-t_j},
\end{equation}
where the infimum is taken over all enumerations $(s_j)$ and $(t_j)$ of $S$ and $T,$ and where $d$ on the right hand side
denotes metric of the space $\mbX.$ The sum on the right hand side will be called \emph{distance between enumerations}.
By ${\bf x}_0$ we denote the element of $\euS_\infty(\mbT)$ which has only one element $x_0$ with infinite multiplicity.
By $\euS_1(\mbX)$ we denote the set of elements of $\euS_\infty(\mbX)$ for which the distance $d(S,{\bf x}_0)$ is finite.

Spectra of unitary operators $U$ of the class $1 + \clL_1(\hilb)$ belong to $\euS_1(\mbT).$
In this paper, this is the only reason for introduction and study of the space $\euS_1(\mbX).$
The only difference is that the eigenvalue $1$ of a unitary operator of the class $1 + \clL_1(\hilb)$ may have a finite multiplicity.
It is possible to treat multiplicity of an eigenvalue in a different way: as the limit of the total multiplicity of all eigenvalues
from a neighbourhood of a given one, when the radius of the neighbourhood goes to zero. This way, the spectrum of a unitary operator of the class
$1+\clL_1$ is exactly an element of $\euS_1(\mbT).$

The space $\euS_1(\mbT)$ (in a different shell) has been studied by other authors as well; for example, in \cite{Pu01FA},
A.\,Pushnitski studies a metric space $X_1$ which is in fact isometric to $\euS_1(\mbT),$ as we shall see.

In \cite{Kato}, T.\,Kato introduces metric between two finite groups $S$ and $T$ of isolated
eigenvalues (counting multiplicities) of the same power $N$ by the formula
$$
  d(S,T) = \max_{j \in 1...N} \abs{s_j-t_j}.
$$
For the space $\euS_1(\mbT)$ the metric (\ref{F: d(S,T) def of}) is more suitable.

\begin{lemma}
  The set $\euS_1(\mbX)$ with distance $d$ is a metric space.
  If $\mbX$ is separable, then $\euS_1(\mbX)$ is also separable.
\end{lemma} \margcom{OK2}
\begin{proof} 
Obviously, the metric $d$ is non-degenerate and is symmetric. So, we need
to check only the triangle inequality $d(S_1,S_3) \leq d(S_1,S_2) + d(S_2,S_3).$
Let $a_1,a_2,\ldots$ be an enumeration of $S_1,$
let $c_1,c_2,\ldots$ be an enumeration of $S_3,$ and let $b_1',b_2',\ldots$ and $b_1'',b_2'',\ldots$
be two enumerations of $S_2.$ Then for some enumeration $c_1',c_2',\ldots$ of $c$
$$
  \sum \abs{a_j-b_j'} + \sum \abs{b_j''-c_j} = \sum \abs{a_j-b_j'} + \sum \abs{b_j'-c_j'} \geq \sum \abs{a_j-c_j'} \geq d(S_1,S_3).
$$
Since the two enumerations of $b$ are arbitrary, it follows that
$$
  d(S_1,S_2) + d(S_2,S_3) \geq d(S_1,S_3).
$$

If $A$ is a countable dense subset of $\mbX,$ then
it is easy to see that the set of all rigged sets in $\euS_1(\mbX),$ which are of finite rank and with elements from $A,$
is countable and is dense in $\euS_1(\mbX).$
\end{proof}

We note some properties of the distance $d.$

\begin{lemma} \label{L: d(S1+S2,T1+T2)<...}
For any $S_1, S_2, T_1, T_2 \in \euS_1(\mbX),$
\begin{equation}
  d(S_1+S_2,T_1+T_2) \leq d(S_1,T_1) + d(S_2,T_2).
\end{equation}
\end{lemma}\margcom{OK2}
\begin{proof} Shuffling of any two enumerations $(s'_j)$ and $(s''_j)$ of $S_1$ and $S_2$ (respectively, of any two enumerations $(t'_j)$ and
$(t''_j)$ of $T_1$ and $T_2$) defines an enumeration $(s_j)$ of $S_1+S_2$ (respectively, an enumeration $(t_j)$ of $T_1+T_2$).
Clearly,
$$
  d(S_1+S_2,T_1+T_2) \leq d\Brs{(s_j),(t_j)} = d\Brs{(s'_j),(t'_j)} + d\Brs{(s''_j),(t''_j)}.
$$
Taking infimum over all enumerations $(s'_j),$ $(s''_j),$ $(t'_j)$ and $(t''_j)$ of $S_1,S_2,T_1$ and $T_2$ respectively, we get the required inequality.
\end{proof}

\begin{cor} If $S_n \to S$ and $T_n \to T$ in $\euS_1(\mbX),$ then $S_n+T_n \to S+T$ in $\euS_1(\mbX).$
\end{cor}

We often identify points of $\mbT$ with points of the interval $[0,2\pi],$ in which $0=2\pi.$

Given $S \in \euS_1(\mbX)$ and $\eps>0,$ we denote by $S(\eps)$ the rigged set of all points $x \in S,$
such that $d(x,x_0) < \eps.$

\begin{lemma} \label{L: S(eps) to S} For any $S \in \euS_1(\mbX)$ \
$
  \lim_{\eps \to 0} S(\eps) = S
$
\ in $\euS_1(\mbX).$
\end{lemma} \margcom{OK}
This is obvious.

\begin{lemma} \label{L: d(S(eps),T(eps)) to ...} For any $S,T \in \euS_1(\mbX)$ \
$
  \lim_{\eps \to 0} d(S(\eps),T(\eps)) = d(S,T).
$
\end{lemma} \margcom{OK}
\begin{proof} This follows from Lemma \ref{L: S(eps) to S} and continuity of the metric $d.$
\end{proof}

The set of all finite rank sets from $\euS_1(\mbX)$ we denote by $\euS_1^o(\mbX).$ Clearly,
\begin{equation} \label{F: S(eps) if finite}
  S(\eps) \in \euS_1^o(\mbX) \ \text{for any} \ S \in \euS_1(\mbX) \ \text{and any} \ \eps>0.
\end{equation}

\subsection{The correspondence $S \leftrightarrow f_S$}
\label{SS: S to fS}

Let $\tilde X_1^o$ be the set of all left-continuous non-increasing functions (see \cite{Pu01FA})
$$
  f \colon \mbT\setminus \set{1} \to \mbZ
$$
with finitely many values. Let $X_1^o = \tilde X_1^o/\sim,$ where $f \sim g,$ if $f-g = \const.$
Let $\pi \colon \tilde X_1^o \to X_1^o$ be the natural projection.

To every $S \in \euS_1^o(\mbT)$ we assign a counting function $f_S \in X_1^o$
by the formula
$$
  f_S(e^{i\theta}) = \set{\# \set{s \in S \colon \theta < s}^* + n, n \in \mbZ}.
$$
If $f \in X_1^o,$ then assigning to every $\theta \in \mbT$ the jump of $f$ at $\theta$ we obtain an element of~$\euS_1^o.$
So, there is a natural one-to-one correspondence
$$
  \euS_1(\mbT) \leftrightarrow X_1^o.
$$

We introduce metric in $X_1^o,$ by the formula (see \cite{Pu01FA})
$$
  \rho_1(f, g) = \inf \set{\int_0^{2\pi} \abs{\tilde f(\theta) - \tilde g(\theta)}\,d\theta, \pi(\tilde f) = f, \pi(\tilde g) = g}.
$$
Note following trivial relations:
\begin{gather*}
  f_{S+T} = f_S + f_T; \\
  \text{if} \ T \leq S, \ \text{then} \ f_{S-T} = f_S - f_T.
\end{gather*}
If $\tilde f, \tilde g \in \tilde X_1^o$ and $f=\pi(\tilde f),$ $g=\pi(\tilde g),$ then there exists $n \in \mbZ,$ such that
$$
  \rho_1(f,g) = \int_0^{2\pi} \abs{\tilde f(\theta) - \tilde g(\theta) + n}\,d\theta.
$$
See \cite[(3.13)]{Pu01FA}.

\subsection{The space $\euS_1(\mbR)$ and the projection $p\colon \euS_1(\mbR) \to \euS_1(\mbT)$}

We say that an element~$S$ of $\euS_1(\mbR) = \euS_1(\mbR,0)$ is positive (negative), if all elements of $S$ are non-negative (non-positive).
We denote by $S_+$ (respectively, $S_-$) the positive (respectively, negative) part of $S \in \euS_1(\mbR),$ which is defined in an obvious way.
By $\abs{S}$ we denote the rigged set $S_+ + (-S_-).$

\begin{lemma} \label{L: S(pm) belong to euS(R)} If $S,T \in \euS_1(\mbR),$ then (a) $d(S,0) = d(S_+,0) + d(S_-,0).$ \ (b) $d(S,T) = d(S_+,T_+) + d(S_-,T_-).$
\end{lemma}
\begin{proof} \margcom{OK}
(b) A pair of enumerations of the pair $(S_+, T_+)$ and that of the pair $(S_-, T_-)$ by shuffling define an enumeration of the pair $(S,T).$
It follows from the definition of $d(S,T),$ that $d(S,T) \geq d(S_+,T_+) + d(S_-,T_-).$

On the other hand, if two enumerations $(s_j)$ and $(t_j)$ of $S$ and $T$ respectively
are given, then for all $j,$ such that one of the numbers
$s_j$ and $t_j$ is negative and the other is positive, we replace the pair $(s_j,t_j)$ by two pairs $(s_j,0)$
and $(0,t_j)$ to get another pair of enumerations with the same distance between enumerations.
It clearly follows that $d(S,T) \leq d(S_+,T_+) + d(S_-,T_-).$ Proof of (b) is complete.

(a) For this, one can set $T=0$ in the equality in (b).
\end{proof}

\begin{cor} \label{C: S converges iff S(pm) converges} A sequence $S_n$ converges in $\euS_1(\mbR)$ if and only if the sequences
$S_{n,\pm}$ converge in $\euS_1(\mbR),$ and $\lim_{n \to \infty} S_n = \lim_{n \to \infty} S_{n,+} + \lim_{n \to \infty} S_{n,-}.$
\end{cor}

\begin{lemma}
If $S \colon [0,1] \to \euS_1(\mbR)$ is a continuous path, such that $S(0) = {\bf 0},$
then $\abs{S}, S_+, S_- \colon [0,1] \to \euS_1(\mbR)$ are also continuous.
\end{lemma}
\begin{proof} Continuity of $\abs{S}$ follows from the obvious inequality $d(\abs{S},\abs{T}) \leq d(S,T).$
Continuity of $S_\pm$ follows from another obvious inequality $d(S_\pm,T_\pm) \leq d(S,T).$
\end{proof}


\begin{lemma} The metric space $\euS_1(\mbR)$ is contractible.
\end{lemma}
\begin{proof} We have to show that there exists a continuous map $F \colon \euS_1(\mbR) \times [0,1] \to \euS_1(\mbR),$
such that $F(S,0) = S$ and $F(S,1) = 0.$ One can take $F(S,t) = tS.$
\end{proof}

Let $p \colon \mbR \to \mbT,$ $p(\theta) = e^{i\theta},$ and let
$$
  p \colon \euS_1(\mbR) \to \euS_1(\mbT)
$$
be the natural projection, induced by the function $p.$
Plainly, the mapping $p$ is continuous (moreover, it is contracting).

\begin{lemma} \label{L: sum is continuous} Let $\clX$ be a metric space and let $S_1, S_2 \colon \clX \to \euS_1(\mbX)$ be two continuous functions.
Then the function $(S_1 + S_2)(x) := S_1(x) + S_2(x)$ is also continuous.
\end{lemma}
\begin{proof} This follows directly from Lemma \ref{L: d(S1+S2,T1+T2)<...}.
\end{proof}

The following example partly explains why it is desirable to assign to element $1$ infinite multiplicity.
\begin{example}
There is a sequence $\set{S_n}$ of elements of $\euS_1(\mbR)$ such that all $S_n$ contain a positive element,
$S_n$ converge to a strictly negative $S,$ but $S$ does not contain $0.$ Example: $S_n = S \cup \set{\frac 1n},$ where $S$ is any strictly negative sequence
from $\euS_1(\mbT).$ Clearly, the distance between $S_n$ and $S$ is $\frac 1n,$ so that $S_n \to S.$
\end{example}

\subsection{An estimate}
The following lemma is intuitively obvious, but its formal proof is quite lengthy,
despite of being elementary. I was able to find neither a short proof
nor a reference. The reader may wish to skip proof of this lemma.
\begin{lemma} \label{L: d(S,T)=dist(S theta,T theta)} Let $S$ and $T$ be two finite elements of $\euS_1(\mbT).$
There exists a pair of increasing enumerations $S^*$ and $T^*$ of $S$ and $T$ such that
$$
  d(S,T) = \dist(S^*,T^*).
$$

\end{lemma} \margcom{OK}
\begin{proof}
(A) We say that a rigged set $S$ is simple, if $\mult(x) \leq 1$ for any $x\in S.$
Observe that, by continuity of the metric $d,$ and by the fact that simple rigged sets are dense in $\euS_1(\mbT),$
it is sufficient to consider the case of simple finite rigged sets $S$ and $T.$
Moreover, for the same reason, we can and do assume that the rigged set $S+T$ is also simple.

(B) We introduce some terminology which will be used only in the proof of this lemma.
We consider elements of $S$ to be positive and elements of $T$ to be negative.
Given enumerations $(s_j)$ and $(t_j)$ of $S$ and $T,$ we call a pair $(s_j,t_j)$ an arrow,
and a pair of enumerations we call a set of arrows. We represent an arrow on the unit circle $\mbT$
by the arc-vector whose length is $\leq \pi.$ In case the length is equal to $\pi,$ we choose an arrow which passes through $1.$
In case of an arrow $(-1,1) \subset \mbT$ we choose the lower arrow to represent the pair.
If an arrow $(s_j,t_j)$ passes through $1$ we agree to split it into to arrows $(s_j,1)$ and $(1,t_j).$

We say that an arrow is positive, if it directs in counter-clockwise direction; otherwise we say that the arrow is negative.

Now, to any pair of enumerations of $S$ and $T$ we assign a set of arrows in the above described way.
These arrows will be called \emph{standard}. Thus, a standard arrow cannot have length larger than $\pi,$
and it cannot pass through $1 \in \mbT.$

Clearly, the distance between two enumerations of $S$ and $T$ is equal to the sum of lengths of all arrows from the corresponding set of arrows.
This sum will be called the mass of the set. We denote sets of arrows by capital Gothic letters $\mathfrak A,$ \ldots
The mass of a set of arrows $\mathfrak A$ we denote by $m(\mathfrak A).$

(C) We say that two arrows are intersecting if their intersection has positive length.
Given a set of arrows $\mathfrak A,$ we apply to it the following operation:
if there is a pair $(s_1,t_1)$ and $(s_2,t_2)$ of intersecting arrows with different signs in the set,
then we replace this pair by the pair $(s_1,t_2)$ and $(s_2,t_1).$

The set of arrows $\mathfrak B,$ thus obtained, represents a pair of enumerations of $S$ and $T,$ and it obviously has less mass: $m(\mathfrak B) \leq m(\mathfrak A).$
Since the rigged sets $S$ and $T$ are finite, this implies that after some finite number of operations of the type just described,
we obtain a set of arrows whose mass is less than or is equal to the mass of the original set and which possess the following property:
there is no pair of arrows of different sign which intersect.

So, from now on we can and do assume that the set of arrows possess this property.

(D) Assume that an arrow $(s,t)$ contains (by (A), necessarily strictly) another arrow $(s',t')$ (by ``strictly'' we mean that the arrows do not have a common end-point).
We say that such a pair of arrows is \emph{bad}.  According to (B), in a bad pair arrows must have the same direction. We agree to replace this pair of arrows
by the pair $(s,t')$ and $(s',t).$ None of the arrows in this new pair contain the other. Such pairs we call to be \emph{good}.

Clearly, this operation of replacing a bad pair of arrows by the good one does not change the mass of a set of arrows.

Claim: after a finite number of operations of this type we obtain a set of arrows which consists of only good pairs.

Proof of the claim. Let $\set{a,b}$ be a bad pair. If we apply the operation to this pair it becomes good.
Given a third arrow $c,$ there are $5+4+3+2+1$ ways it can be positioned with respect to the arrows $a$ and $b.$
A straightforward and simple check shows that the number of bad pairs in the set $\set{\set{a,c},\set{b,c}}$ does not increase.
It follows that the number of bad pairs decreases after each operation. The claim is proved.

(E) As a result, for any pair $\mathfrak A$ of enumerations of $S$ and $T$ (that is, a set of arrows) there exists
another set of arrows $\mathfrak B$ with mass less or equal to that of $\mathfrak A,$ and such that any pair of arrows in $\mathfrak B$ is good,
(that is, any pair of arrows either do not intersect or, if they intersect, they have the same direction and do not contain one another).
Such a set of arrows will be called a good one.

So, from now on we can and do assume that the set of arrows is a good one.

(F) We introduce an equivalence relation in a good set of arrows: two arrows $a$ and $b$ are equivalent, if there exists a finite sequence
of good pairs $\set{a,c_1}, \set{c_1,c_2}, \ldots, \set{c_n,b}.$ Clearly, this is an equivalence relation. The union of all (closed) arrows in an equivalence class
will be called an \emph{island}; we say that an arrow belongs to an island, if it is a subset (here I am not pedantic in the use of set-theoretic language) of the island.
Thus, an island is a connected (closed) arc. The distance between any two islands is necessarily positive.

(G) Claim: for any good set of arrows there exist enumerations
$$
  0 \leq \ldots \leq 0 < s_{-k} \leq \ldots \leq s_n < 2\pi \leq \ldots \leq 2\pi
$$
and
$$
  0 \leq \ldots \leq 0 < t_{-l} \leq \ldots \leq t_m < 2\pi \leq \ldots \leq 2\pi
$$
of $S$ and $T$ of the same length, which correspond to the given set of arrows (number of zeros
and of $2\pi$'s in the enumerations can be different).

Proof. We order islands in increasing order.
For every island, we enumerate initial points of the arrows, which belong to the island, in increasing order.
It is not difficult to see that terminal points of the arrows will become enumerated in increasing order as well.

Next, we juxtapose enumerations of islands and get the required enumerations of $S$ and $T.$

The claim is proved.

(H) We have proved so far, that in the equality
$d(S,T) = \inf_{S^*,T^*} \dist(S^*,T^*),$
where the infimum is taken over all enumerations $S^*$ and $T^*,$
the infimum can be replaced by infimum over all increasing enumerations $S^*$ and $T^*.$
Since the number of different pairs of increasing enumerations (which define a good set of arrows)
is obviously finite, the equality $d(S,T) = \dist(S^*,T^*)$ holds for such a pair.

The proof of the lemma is complete.
\end{proof}

For the notation, used in the following corollary, see subsection \ref{SS: S to fS}.
\begin{cor} \label{C: d = rho1} For any $S,T \in \euS_1^o(\mbT),$
\begin{equation} \label{F: d = rho}
  d(S,T) = \rho_1(f_S,f_T).
\end{equation}
\end{cor} \margcom{OK2}
\begin{proof}
(A) Every representative of $f_S$ is a counting function of some increasing enumeration of $S.$
This is obvious.

(B) For any pair of increasing enumerations $S^*$ and $T^*$ of $S$ and $T$
$$
  \dist(S^*,T^*) = \int_0^{2\pi} \abs{f_{S^*}(\theta) - f_{T^*}(\theta)}\,d\theta,
$$
where $f_{S^*}$ is a counting function of $S^*.$ \margdetails

Proof. It is clear that $\dist(S^*,T^*)$ is equal to the integral of the absolute value of the current function $\mu$
of the pair which is equal to $f_{S^*}(\theta) - f_{T^*}(\theta).$

(C) It follows from (A) and (B) that
$$
  \rho_1(f_S,f_T) = \inf _{S^*,T^*} \dist(S^*,T^*),
$$
where the infimum is taken over all increasing enumerations of $S$ and $T.$
It follows from this and Lemma \ref{L: d(S,T)=dist(S theta,T theta)} that (\ref{F: d = rho}) holds.
\end{proof}

\begin{rems} Corollary \ref{C: d = rho1} shows that the metric space $\euS_1(\mbT)$
is in fact isometric to the metric space $X_1$ introduced in \cite{Pu01FA}.
\end{rems}

Note that $S_1 \leq S_2$ and $T_1 \leq T_2$ do not imply $d(S_1,T_1) \leq d(S_2,T_2);$
otherwise, $d(S,T) \leq d(S+T,S+T) = 0.$
But the following inequality holds. It is possible that this inequality holds in any metric space.
I was not able to find a reference for it or prove it. Anyway, we don't need this stronger result in this paper.

\begin{thm} \label{T: important estimate}
For any rigged sets $S_1, S, T_1, T \in \euS_1(\mbT),$ if $S_1 \leq S$ and $T_1 \leq T,$ then
\begin{equation}
  d(S - S_1, T - T_1) \leq d(S_1,T_1) + d(S,T).
\end{equation}
\end{thm} \margcom{OK2}
\begin{proof}
(A) \ Claim: the inequality holds in the case of finite rigged sets $S_1, S, T_1, T \in \euS_1(\mbT).$

Proof. By Corollary \ref{C: d = rho1}, it is enough to show that
$$
  \rho_1(f_{S - S_1}, f_{T - T_1}) \leq \rho_1(f_{S_1}, f_{T_1}) + \rho_1(f_S,f_T).
$$
We have
\begin{equation*}
  \begin{split}
  \rho_1(f_{S - S_1}, f_{T - T_1}) & = \rho_1(f_{S} - f_{S_1}, f_{T} - f_{T_1})
  \\ & = \inf_{m,n\in\mbZ} \int_0^{2\pi} \abs{\brs{f_{S}(\theta) - f_{S_1}(\theta)+m} - \brs{f_{T}(\theta) - f_{T_1}(\theta)+n}}\,d\theta
  \\ & \leq \inf_{m\in\mbZ} \int_0^{2\pi} \abs{f_{S_1}(\theta) - f_{T_1}(\theta)+m}\,d\theta + \inf_{m\in\mbZ} \int_0^{2\pi} \abs{f_{S}(\theta) - f_{T}(\theta)+n}\,d\theta
  \\ & = \rho_1(f_{S_1}, f_{T_1}) + \rho_1(f_S,f_T).
  \end{split}
\end{equation*}

(B) It follows from (\ref{F: S(eps) if finite}) and (A), that for any $\eps>0$
\begin{equation*}
  d(S(\eps) - S_1(\eps), T(\eps) - T_1(\eps)) \leq d(S_1(\eps),T_1(\eps)) + d(S(\eps),T(\eps)).
\end{equation*}
Since $(A-B)(\eps) = A(\eps)-B(\eps)$ for any $A,B \in \euS_1(\mbT),$
it follows that
\begin{equation*}
  d((S-S_1)(\eps), (T-T_1)(\eps)) \leq d(S_1(\eps),T_1(\eps)) + d(S(\eps),T(\eps)).
\end{equation*}
Now, by Lemma \ref{L: d(S(eps),T(eps)) to ...}, taking the limit $\eps \to 0$ completes the proof.
\end{proof}

\begin{lemma} \label{L: difference is continuous} If $\clX$ is a metric space and if $S,T \colon \clX \to \euS_1(\mbT)$ (or $\euS_1(\mbR)$)
are two continuous functions such that $S(x) \leq T(x)$ for all $x \in \clX,$ then the function
$T-S \colon \clX \to \euS_1(\mbT)$ (or $\euS_1(\mbR)$) is continuous.
\end{lemma} \margcom{OK}
\begin{proof} By Theorem \ref{T: important estimate}, for any $x_0, x \in \clX,$
$$
  d((T-S)(x_0),(T-S)(x)) \leq d(T(x_0),T(x)) + d(S(x_0),S(x)).
$$
\end{proof}

\subsection{Completeness of $\euS_1(\mbT)$}


\begin{lemma} \label{L: euS(mbR) is complete} The metric space $\euS_1(\mbR)$ is complete. \margcom{OK}
\end{lemma}
\begin{proof} (A) The metric space $\euS_1(\mbR_+)$ is complete.
Proof. We observe that the distance between two rigged sets $S$ and $T$ from $\euS_1(\mbR_+)$
is equal to the distance between non-increasing enumerations of $S$ and $T$
(this assertion is either obvious or can be proved using the same argument as in the proof of Lemma \ref{L: d(S,T)=dist(S theta,T theta)}).
Hence, the claim follows from completeness of $l_1.$

(B) If a sequence $\set{S_n}$ is Cauchy in $\euS_1(\mbR),$ then, by Lemma \ref{L: S(pm) belong to euS(R)}(b),
the sequences $\set{S_{n,\pm}}$ are Cauchy in $\euS_1(\mbR_+).$ By (A), $\set{S_{n,\pm}}$ converges to some $S_\pm$ in $\euS_1(\mbR_+).$
By Lemma \ref{L: S(pm) belong to euS(R)}(b), the sequence $\set{S_n}$ converges to $S_+ + S_-.$
\end{proof}

Given $S \in \euS_\infty(\mbT),$ we define the sequence of finite sets $S^{(1)}, S^{(2)}, S^{(3)}, \ldots$
as follows: we enumerate elements of $S$ so that distances of elements of $S$ to $1$ form a decreasing sequence,
and $S^{(N)}$ is by definition the set which contains the first $N$ elements of the enumeration.
In case when there are mutually conjugate elements we agree to take firstly the elements with positive imaginary part.
In this way, the sequence $S^{(1)}, S^{(2)}, S^{(3)}, \ldots$ becomes uniquely determined by $S.$

We define $\bar S^{(N)}$ as the complement of $S^{(N)}$ in $S.$

If $S \in \euS_1(\mbT),$ then obviously $$S^{(N)} \to S \ \text{in} \ \euS_1(\mbT)$$
and
$$
  d(\bar S^{(N)},{\bf 1}) \to 0
$$
as $N \to \infty.$

Let $z \in \mbT.$ We define multiplicity $k\in \set{0,1,2,\ldots,\infty}$ of $z$ with respect to a Cauchy sequence $\set{S_n}$ as follows:
there exists (sufficiently small) $\eps_0>0$ such that for all $\eps \in (0,\eps_0)$
there exists $N \in \mbN$ such that for all $n \geq N$
the $\eps$-neighbourhood of $z$ contains exactly $k$ elements of $S_n.$

\begin{lemma} \label{L: any x has mult-ty} Every point $z \in \mbT$ has multiplicity.
The rigged set of points with positive multiplicity belongs to $\euS_\infty(\mbT).$
\end{lemma} \margcom{OK2}
\begin{proof} Obviously, the point $z = 1$ has multiplicity $\infty:$ one can take $\eps_0=1$ and $N=1$
in the definition.
So, assume that $z \neq 1.$
It is clear that multiplicity of a point $z \neq 1$ cannot be $\infty,$
since otherwise $z$ will be accumulation point of some $S_n.$

Let $U$ be a neighbourhood of $z$ which is separated from $1.$
Since $\set{S_n}$ is a Cauchy sequence, the distance from $S_n$ to ${\bf 1}$ is uniformly bounded.
It follows that there exists $k_0$ such that for all $n$ the number of elements of $S_n$
in $U$ is $\leq k_0.$ Since $\set{S_n}$ is Cauchy, it follows that the number of accumulation points of the sum
$\sum S_n$ in $U$ cannot be larger than $k_0.$

It follows that the set of accumulation points $S$ of the sum $\sum S_n$ cannot have accumulation points other than $1.$

Obviously, non-accumulation points of $\sum S_n$ have multiplicity $0.$

Now, let $z$ be an accumulation point of $\sum S_n$ not equal to $1;$ that is, let $1\neq z \in S.$ Choose a neighbourhood $U$ of $z$ which is separated from
other elements of $S.$ Again, since $\set{S_n}$ is Cauchy, the number $k$ of elements of $S_n$
in $U$ must stabilize as $n \to \infty.$ Since $z$ is the only accumulation point of $\sum S_n$ in $U,$
this number $k$ is the multiplicity of $z.$
\end{proof}

\begin{lemma}
  The rigged set of points $z$ with positive multiplicity belongs to $\euS_1(\mbT).$
\end{lemma} \margcom{OK2}
\begin{proof}
By Lemma \ref{L: any x has mult-ty}, the rigged set $S$ of points with positive multiplicity belongs to $\euS_\infty(\mbT).$

Let $N \in \mbN.$
For any $\eps>0$ there exists $S_n$ such that the distance between some rigged subset $A$ of $S_n$
and $S^{(N)}$ is less than $\eps.$ It follows from this and Lemma \ref{L: d(S1+S2,T1+T2)<...} that
\begin{equation*}
  \begin{split}
    d(S^{(N)},S_n) \leq d(S^{(N)},A) + d({\bf 1},S_n - A) < \eps + C,
  \end{split}
\end{equation*}
where $C = \sup d(S_n,{\bf 1})$ is finite, since $\set{S_n}$ is Cauchy. It follows from this and the triangle inequality that
$$
  d(S^{(N)},{\bf 1}) \leq d(S^{(N)},S_n) + d(S_n,{\bf 1}) < \eps + 2C,
$$
so that
$$
  d(S,{\bf 1}) = \sup d(S^{(N)},{\bf 1}) \leq 2C.
$$
Hence, $S \in \euS_1(\mbT).$
\end{proof}

\begin{thm}
  The metric space $\euS_1(\mbT)$ is complete.
\end{thm}
\begin{proof}
We shall show that any Cauchy sequence $\set{S_n}$ converges in $\euS_1(\mbT)$ to the rigged set $S$
of points with positive multiplicity.

Let $S_n'$ be the (finite) part of $S_n$ in the half-circle $[i,-i]$ and $S_n''$ be
the part of $S_n$ in the half-circle $(-i,i).$
Let $U'$ be a neighbourhood of $S'$ --- the part of $S$ in $[i,-i]$, and let $U''$ be a neighbourhood of $(-i,i)$
such that the distance between $U'$ and $U''$ is $\delta > 0$ (clearly, such neighbourhoods exist). 

Plainly, $S_n'$ converges to $S'.$ The sequence $S_n''$ is obviously also Cauchy, so by Lemma \ref{L: euS(mbR) is complete}
it converges to some $S''.$ It can be easily seen that for any $\eps>0$ $S''(\eps) = S(\eps) \cap [-i,i].$
It follows that $S = S'+S''.$ Consequently, the sequence $S_n = S_n'+S_n''$ converges to $S$ (by Lemma \ref{L: d(S1+S2,T1+T2)<...}).
\end{proof}

\subsection{Characterization of continuous maps $S \colon [0,1] \to \euS_1(\mbT)$}
Assume that there is a sequence of continuous functions $z_1(\cdot), z_2(\cdot), \ldots$ on $[0,1]$ which take values
in $\mbT,$ and such that the rigged set $\set{z_1(r),z_2(r),\ldots}^*$ belongs to $\euS_1(\mbT)$
for all $r \in [0,1].$ So, there is a function $S \colon [0,1] \to \euS_1(\mbT).$
A natural question is whether this function is continuous or not.
The answer is negative as the following example shows.
\begin{example} Write in a sequence all functions of the form $n^{-2}f(2^k x),$ $k,n \in \mbN,$ where $f(x) = -\sin x,$ if $x \in [\pi,2\pi],$ and $f(x) = 0,$
if otherwise. These functions are continuous, there values at each $x \in [0,1]$ form a rigged set which belong to $\euS_1(\mbR),$ still the corresponding function
$[0,1] \to \euS_1(\mbR)$ is not continuous at $0.$
\end{example}

It can be seen that the complete separable metric space $\euS_1(\mbT)$ is not compact.
For example, the sequence $\brs{S_n = \set{1,1/2,\ldots,1/n}}$ does not have a convergent subsequence.
But $\euS_1(\mbT)$ has the following property.
\begin{prop} \label{P: if K is compact ...} If $K$ is a compact subset of $\euS_1(\mbT),$ then
$$
  \lim_{\eps \to 0} \sup_{S \in K} d(\bar S(\eps), {\bf 1}) = 0.
$$
\end{prop} \margcom{OK}
This is, in fact, a property of $l_1.$
\begin{proof} The set $K_\eps$ of all $S \in K$ such that $\supp S \subset [e^{-i\eps},e^{i\eps}] \subset \mbT$
is closed, and so is compact. 

Since the function of $\eps$ on the left hand side is decreasing, its limit $\alpha$ exists.
There exists a sequence $S_n \in K_{1/n}$ such that $d(S_n, {\bf 1}) \to \alpha.$
Since $K$ is compact, we can assume that $S_n$ converges to some $S \in K.$ It follows that $d(S_n, {\bf 1}) \to d(S, {\bf 1}) = \alpha.$
Since obviously $\supp S \subset \cap [-1/n,1/n],$ it follows that $S = {\bf 1},$ and so $\alpha = 0.$
\end{proof}

For a continuous function $S \colon [0,1] \to \euS_1(\mbT)$ with a continuous enumeration $\set{f_1,f_2,\ldots}^*,$
Proposition \ref{P: if K is compact ...} gives a necessary condition which the continuous enumeration must satisfy.
The following lemma shows that this necessary condition is also sufficient.
\begin{lemma} \label{L: cont-s f-ns imply cont-s set f-n}
Let $\clX$ be a metric space.
If $z_n \in C(\clX,\mbT),$ $n=1,2,\ldots,$ if the rigged set
$S(x) := \set{z_1(x), z_2(x), \ldots}^* \in \euS_1(\mbT)$ for any $x \in \clX$ 
and if
\begin{equation} \label{F: 1}
  \lim_{\eps \to 0} \sup _{x\in \clX} d\brs{\bar S(x)(\eps),{\bf 1}} = 0,
\end{equation}
then $S \colon \clX \to \euS_1(\mbT)$ is continuous.
\end{lemma} \margcom{OK}
\begin{proof}
%

Let $x_0 \in \clX.$ We show that $S$ is continuous at $x_0.$
Let $\delta > 0.$ From (\ref{F: 1}) it follows that there exists $\eps_0>0$ such that for any $x \in \clX$ and any $\eps\leq \eps_0$
\begin{equation} \label{F: 10}
  d(\bar S(x)(\eps),{\bf 1}) < \delta/3.
\end{equation}
Replacing $\eps_0$ by a slightly smaller number, if necessary, we can assume that for any $z \in S(x_0)(\eps_0)$
the distance between $z$ and $1$ is strictly larger than $\eps_0.$

By Lemma \ref{L: d(S1+S2,T1+T2)<...}, the triangle inequality and (\ref{F: 10}), for any $x \in \clX$
\begin{equation} \label{F: 11}
 \begin{split}
  d(S(x),S(x_0)) & \leq d(S(x)(\eps_0),S(x_0)(\eps_0)) + d(\bar S(x)(\eps_0), \bar S(x_0)(\eps_0))
  \\ & < d(S(x)(\eps_0),S(x_0)(\eps_0)) + 2\delta/3.
 \end{split}
\end{equation}
By Lemma \ref{L: rank on U is const} (clearly, we can assume that $\clX$ is connected), there exists a neighbourhood $W_0$ of $x_0$ and
there exists $N = N(\eps_0),$ such that for all $x \in W_0$ the number of elements of $S(x)(\eps_0)$ is equal to $N.$
For any element $z_j(x_0), \ j=1,2,\ldots,N,$ of $S(x_0)(\eps_0)$ there exists a neighbourhood $W_j$ of $x_0$ such that for all $x \in W_j$
$$
  \abs{z_j(x_0)-z_j(x)} < \frac{\delta}{3N}.
$$
It follows that for all $x \in W:=\bigcap_{j=0}^N W_j,$
$$
  d(S(x)(\eps_0), S(x_0)(\eps_0)) \leq \sum_{j=1}^N \abs{z_j(x) - z_j(x_0)} < \delta/3.
$$
Combining this with (\ref{F: 11}) we obtain that $d(S(x), S(x_0)) < \delta$ for any $x \in W.$
\end{proof}

\begin{cor} \label{C: subset of enumeration is continuous}
Let $\clX$ be a metric space.
If $\set{z_1, z_2, \ldots}^*$ is a continuous enumeration of a continuous function $S \colon \clX \to \euS_1(\mbT),$
then any rigged subset of $\set{z_1, z_2, \ldots}^*$ determines a continuous function $\clX \to \euS_1(\mbT)$ (which is clearly $\leq S$).
\end{cor} \margcom{OK}

\subsection{Some lemmas}

\begin{lemma} \label{L: the sum is continuous}
If $\clX$ is a metric space and if \ $\tilde S \colon \clX \to \euS_1(\mbR)$ \ is a continuous function,
then the function $\clX \ni x \mapsto \sum \tilde S(x) \in \mbR$ is continuous.
\end{lemma} \margcom{OK}
\begin{proof} If $f(x) = \sum \tilde S(x),$ then for any enumerations $(\theta_1(x),\theta_2(x),\ldots)$
and $(\theta_1(x_0),\theta_2(x_0),\ldots)$ of $\tilde S(x)$ and $\tilde S_0(x).$
$$
  \abs{f(x) - f(x_0)} \leq \abs{\sum \tilde S(x) - \sum \tilde S(x_0)} \leq \sum_{j=1}^\infty\abs{\theta_j(x) - \theta_j(x_0)}.
$$
It follows that $\abs{f(x) - f(x_0)} \leq d(\tilde S(x), \tilde S(x_0)).$ Hence, $f$ is continuous.
\end{proof}

This lemma shows that if $\tilde S \colon [0,1] \to \euS_1(\mbR)$ is continuous and $\set{\theta_j(x)}$
is a continuous enumeration of $\tilde S(x),$ then the function
$$
  \sum_{j=1}^\infty \abs{\theta_j(x)}
$$
must be continuous (since $|\tilde S|$ is also continuous), and, consequently, it also must be bounded.

\begin{lemma} \label{L: the sum for different ligtings are the same}
Let $\clX$ be an arc-wise connected metric space and let $x_0 \in \clX.$
If $\tilde S_1$ and $\tilde S_2$ are two different liftings of a continuous function $S \colon \clX \to \euS_1(\mbT),$
such that $\tilde S_1(x_0) = \tilde S_2(x_0) = 0,$
then for all $x \in \clX$ \margcom{OK2}
$$
  \sum \tilde S_1(x) = \sum \tilde S_2(x).
$$
\end{lemma}
\begin{proof} Let $f_j(x) = \sum \tilde S_j(x),$ $j=1,2.$ By Lemma \ref{L: the sum is continuous}, functions $f_1$ and $f_2$
are continuous; also, $f_1(x_0) = f_2(x_0).$

Let $x_1 \in \clX$ and let $x_r, r \in [0,1],$ be a continuous path which connects $x_0$ and $x_1.$
Continuous functions $f_1(x_r)$ and $f_2(x_r)$ must differ by an integer multiple of $2\pi,$ since $p\circ f_1 = p \circ f_2.$
It follows that $f_1(x_r) = f_2(x_r)$ for all $r \in [0,1],$ and in particular, $f_1(x_1) = f_2(x_1).$
\end{proof}

\begin{lemma} \label{L: sum theta converges uniformly}
Let $\clX$ be a compact metric space.
Let $\tilde S \colon \clX \to \euS_1(\mbR)$ be a continuous function.
If $\theta_1, \theta_2, \ldots$ is a continuous enumeration of $\tilde S,$
then the series
$$
  \sum_{j=1}^\infty \theta_j(x)
$$
is uniformly convergent.
\end{lemma} \margcom{OK}
\begin{proof} (A) Assume first that $\tilde S \geq 0.$
In this case, the sequence of partial sums of the above series
is increasing, consists of continuous functions and converges to a continuous function (by Lemma \ref{L: the sum is continuous}).
By \cite[Lemma 6.13]{Az3v4} (in that lemma $\clX =[0,1],$ but this is not essential for its proof),
it follows that the series converges uniformly.

(B) A continuous enumeration $\theta_1, \theta_2, \ldots$ of $\tilde S$ can be replaced by the continuous enumeration
$\theta_1^+,\theta_{1}^-,\theta_{2}^+,\theta_{2}^-,\ldots$ of $\tilde S,$
where $\theta_j^+ = \max(\theta_j,0)$ and $\theta_j^- = \min(\theta_j,0).$
 By (A), for enumerations $\theta_{1}^\pm,\theta_{2}^\pm,\theta_{3}^\pm,\ldots$ of $\tilde S_\pm$
(which also belong to $\euS_1(\mbT),$ by Lemma \ref{L: S(pm) belong to euS(R)}),
the series converges uniformly. It follows that the given series converges uniformly.
\end{proof}

\section{Selection Theorem}
The aim of this section is to prove the following selection theorem
(to the best of my knowledge, this theorem is new).
As is noted in \cite[Remark VII.3.11]{Kato}, these kind of theorems, despite of being intuitively obvious,
are non-trivial, --- even in the case of spectra of a holomorphic family of compact operators.
\begin{thm} \label{T: Selection Thm}
Let \ $S \colon [0,1] \to \euS_1(\mbT)$ \ be a continuous path, such that $S(0) = {\bf 1}.$
There exists a sequence of continuous functions
$$
  z_j\colon [0,1] \to \mbT, \ \ j=1,2,\ldots
$$
such that $z_j(0) = 1$ for all $j=1,2,\ldots$ and the sequence $(z_1(r), z_2(r), \ldots)$ is an enumeration of $S(r)$ for every $r \in [0,1].$
\end{thm}


The idea of the proof is to lift the function $S \colon [0,1] \to \euS_1(\mbT)$ to $\euS_1(\mbR),$
find a continuous enumeration of the lifting and then to project it back to $\euS_1(\mbT).$

The following show, that once it is shown that this lifting exists, then
it is a simple matter to find a continuous enumeration of the lifting.

\begin{lemma} \label{L: continuous sections of path S}
Let $\clX$ be a metric space and let \ $\tilde S \colon \clX \to \euS_1(\mbR)$ \ be a continuous function.
Let
$$
  \theta_1^+(x) = \sup \tilde S(x), \ \ldots, \theta_n^+(x) = \sup \set{\tilde S(x) - \set{\theta_1^+(x), \ldots, \theta_{n-1}^+(x)}^*}, \ \ldots,
$$
and
$$
  \theta_{1}^-(x) = \sup \tilde S(x), \ \ldots, \theta_{n}^-(x) = \sup \set{\tilde S(x) - \set{\theta_1^-(x), \ldots, \theta_{n-1}^-(x)}^*}, \ \ldots.
$$
All functions $\theta_1^+, \theta_1^-, \theta_2^+, \theta_2^-, \ldots$ are continuous and the rigged set
$$
  \set{\theta_1^+(x),\theta_1^-(x),\theta_2^+(x),\theta_2^-(x),\ldots}^*+{\bf 0}
$$ coincides with $S(x).$
\end{lemma} \margcom{OK}
\begin{proof}
Plainly, the rigged set $\set{\theta_1^+(x),\theta_2^+(x),\ldots}^*$ coincides with positive part of $\tilde S(x)$ (up to $\set{0}$);
similarly, the rigged set $\set{\theta_1^-(x),\theta_2^-(x),\ldots}^*$ coincides with negative part of $\tilde S(x)$ (up to $\set{0}$).
It is not difficult to see that for any $x_1, x_2 \in \clX,$
$$
  \abs{\theta_1^+(x_1) - \theta_1^+(x_2)} \leq d(\tilde S(x_1),\tilde S(x_2)),
$$
so that $\theta_1^+$ is continuous. It follows from this and Lemma \ref{L: difference is continuous}, that $\theta_2^+$ is continuous, etc.
Similarly, $\theta_1^-, \theta_2^-, \ldots$ are also continuous.
\end{proof}

The next two subsections will present some necessary preparatory material essential for the proof of Theorem \ref{T: Selection Thm}.

\subsection{Reducing open sets}

We introduce the following definition for convenience.
\begin{defn} Let $\clX$ be a metric space and let $S \colon \clX \to \euS_1(\mbT)$ or $\euS_1(\mbR)$ be continuous.
We say that an open subset $U_1$ of $\mbT$ (or $\mbR$) is \emph{reducing} for $S$ \emph{on a subset} $K$ of $\clX,$ if there exists another open subset $U_2$
such that the distance between $U_1$ and $U_2$ is positive and $\supp S(x) \subset U_1 \cup U_2$ for all $x \in K.$
If $K=\clX$ then we say that $U_1$ is \emph{reducing} for~$S.$

If two open sets $U_1$ and $U_2$ satisfy the above conditions then we say that $U_1$ and $U_2$ is a pair of reducing open sets.

If $U_1$ is reducing for $S,$ then \emph{restriction} of $S$ to $U_1$ is the function $S\big|_{U_1} \colon \clX \to \euS_1(\mbT)$ (or $\euS_1(\mbT)$)
defined by the formula
\begin{center}
  $(S\big|_ {U_1})(x) := S(x) \cap U_1.$
\end{center}
\end{defn}

For a pair of reducing open subsets $U_1$ and $U_2$ of $\mbT$ we have
$$
  S = S \big|_{U_1} + S \big|_{U_2}.
$$

The following lemmas are intuitively obvious and their proofs are elementary. The reader may wish to skip them.

\begin{lemma} \label{L: continuity of restriction} Let $\clX$ be a metric space, let $S \colon \clX \to \euS_1(\mbT)$ or $\euS_1(\mbR)$
be continuous. If an open subset $U$ of $\mbT$ (or $\mbR$) is reducing, then $S\big|_U$ is continuous.
\end{lemma} \margcom{OK}
\begin{proof}
Let $\eps>0$ be the distance between $U$ and some other open set $U_2$ such that $\supp S \subset U\cup U_2.$
Let $x_0 \in \clX.$
One of the sets $U$ or $U_2$ contain $1$ (or $0$). It follows that one of the rigged sets $S(x_0) \cap U$ and $S(x_0) \cap U_2$
is finite. Assume first, that $S(x_0) \cap U$ is finite.
Since $S$ is continuous, there exists a neighbourhood $W$ of $x_0,$ such that
$$
  d(S(x),S(x_0)) < \eps
$$
for all $x$ from $W.$
This means that for all $x \in W$ there exist enumerations of $S(x_0)$ and $S(x)$ such that the distance between enumerations is less than $\eps.$
It follows from $d(S(x),S(x_0)) < \eps$ that for all $x \in W$ the number of elements in $S(x) \cap U$ must be the same as that of $S(x_0) \cap U$
and that the distance between $S(x_0) \cap U$ and $S(x) \cap U$ must be less $\eps.$ The last means continuity of $S\big|_U.$

Now, if it is the rigged set $S(x_0) \cap U_2$ which is finite, then by the above the function $S\big|_{U_2}$ is continuous.
It follows from Theorem \ref{T: important estimate} that $S\big|_{U} = S - S\big|_{U_2}$ is also continuous.
\end{proof}

Given two rigged sets $A$ and $B$ such that $\supp A \cap \supp B = \emptyset,$ we write $A \sqcup B$ instead of $A+B.$

\begin{lemma} \label{L: X1 and X2} Let $\clX$ be a metric space and let $f \colon \clX \to \euS_1(\mbT)$ be a continuous function.
Let $x_0 \in \clX$ and let $f(x_0) = X_1 \sqcup X_2,$ where $X_2$ is finite.
Then there exists a neighbourhood $W$ of the point $x_0,$ and there exist open sets $U_1 \supset \supp X_1,$
$U_2 \supset \supp X_2,$ such that the pair $U_1$ and $U_2$ is reducing for $f$ on $W.$
\end{lemma} \margcom{OK}
\begin{proof} Let $U_1$ and $U_2$ be a pair of open sets, such that $U_1 \supset \supp X_1$ and $U_2 \supset \supp X_2,$
and such that the distance between $U_1$ and $U_2$ is positive. Since $\supp f(x_0) \subset U_1 \cup U_2$ and since $f$ is continuous, there exists a neighbourhood $W$ of the point $x_0,$
such that for any $x \in W$ the inclusion $f(x) \subset U_1 \cup U_2$ holds. So, the pair $U_1$ and $U_2$ is reducing for $f$ on $W.$
\end{proof}

\begin{lemma} \label{L: continuous liftings of reduced parts} Let $\clX$ be a metric space and let $S \colon \clX \to \euS_1(\mbT)$ be a continuous function.
Let $U_1$ and $U_2$ be a pair of reducing open subsets of $\mbT$ for $S$ on $\clX.$
For any continuous lifting $\tilde S \colon \clX \to \euS_1(\mbR)$ of $S$
there exist (unique) continuous liftings $\tilde S_1$ and $\tilde S_2$ of $S\big|_{U_1}$ and $S\big|_{U_2}$
such that $\tilde S = \tilde S_1 + \tilde S_2.$
\end{lemma} \margcom{OK}
\begin{proof} Let $\tilde S_j = \tilde S \cap p^{-1}(U_j).$ Clearly, $\tilde S = \tilde S_1 + \tilde S_2$
and the pair of open subsets $p^{-1}(U_1)$ and $p^{-1}(U_2)$ is reducing for $\tilde S.$
It follows from Lemma \ref{L: continuity of restriction} that $\tilde S_1$ and $\tilde S_2$ are continuous.
\end{proof}

\begin{lemma} \label{L: rank on U is const}
Let $\clX$ be a connected metric space, let $S \colon \clX \to \euS_1(\mbT)$ (or $\euS_1(\mbR)$)
be a continuous function and let $U_1$ be an open subset of $\mbT$ (or $\mbR$) such that the distance between $U_1$
and $1$ (or $0$) is $>0.$ If $U_1$ is reducing for $S,$ then
$$
  \rank S\big|_{U_1} = \const.
$$
\end{lemma}
\begin{proof} Assume the contrary: $\rank S(x)$ takes at least two values $N_1$ and $N_2.$
Let $U_2$ be an open subset such that the pair $U_1$ and $U_2$ is reducing and let $\eps = \dist(U_1, U_2) >0.$
Since $\clX$ is connected, the sets
$$
  K_1 = \set{x \in \clX \colon \rank S(x) = N_1} \ \text{and} \ K_2 = \set{x \in \clX \colon \rank S(x) \neq N_1}
$$
have intersecting closures.  Assume that some $x_0 \in K_1 \cap K_2$ belongs to $K_2.$
It follows that there exists a sequence $x_1, x_2, \ldots$ of elements of $K_1$ converging to an element $x_0 \in K_2.$
Since $S$ is continuous, there exists $N$ such that for all $n \geq N$ \ \ $d(S(x_n),S(x_0)) < \eps.$
Since $\dist(U_1,U_2) = \eps$ and since $\rank(S(x_n)) \neq \rank(S(x_0)),$ we get a contradiction.
\end{proof}

\begin{lemma} \label{L: one-point lift}
Let $\clX$ be a metric space and let $x_0 \in \clX.$
Let $U$ be a reducing open subset of $\mbT$ for a continuous function $S \colon \clX \to \euS_1(\mbT),$ such that $\dist(U,1) > 0$
and let $\supp S(x_0) = \set{e^{i\theta_0}} \in U.$ If $\tilde S_0 \in \euS_1(\mbR)$ is such that
$p \tilde S_0 = S(x_0),$ then there exists a continuous lifting $\tilde S \colon \clX \to \euS_1(\mbR)$
of $S$ such that $\tilde S(x_0) = \tilde S_0.$
\end{lemma} \margcom{OK}
\begin{proof} Since $\supp S(x_0) \subset U,$ and since $U$ is reducing for $S,$ it follows that $\supp S(x) \subset U$ for all $x \in \clX.$
Let $\tilde U$ be a subset of $(0,2\pi)$ which corresponds to $U$ under the map $p,$ so that $p^{-1} \colon U \simeq \tilde U.$
Since $\dist(U,1)>0,$ the function $\tilde T \colon \clX \to \euS_1(\mbR),$ \ $\tilde T = p^{-1} \circ S,$ thus obtained,
is continuous. By Lemma \ref{L: continuous sections of path S}, the function $\tilde T$ admits a continuous enumeration
$\theta_1, \ldots, \theta_N.$ Now, all we need to do is to add to each of these functions one of the numbers $\set{ 2\pi \left[\frac \theta {2\pi}\right] \colon \theta \in \tilde S_0}^*.$
Plainly (or by Lemma \ref{L: cont-s f-ns imply cont-s set f-n}), values of the resulting functions determine a continuous function $\euS_1(\mbR)$ with required properties.
\end{proof}

\subsection{Prolongation to $[0,r_0]$}
\begin{lemma} \label{L: prolongation of finite paths}
  Let $U$ be an open subset of $\mbT$ such that the distance between $U$ and $1$ is not zero. Let $r_0>0.$
  Let $S \colon [0,r_0] \to \euS_1(\mbT)$ be a continuous path such that $\supp S(r) \subset U$ for all $r \in [0,r_0].$
  If there exists a continuous lifting $\tilde S \colon [0,r_0) \to \euS_1(\mbR)$ of $S$
  on $[0,r_0),$ then this lifting can be continuously prolonged to $[0,r_0].$
\end{lemma} \margcom{OK}
\begin{proof} We prove this lemma by induction. Let $N$ be the number of elements in $S(r).$
By Lemma \ref{L: rank on U is const}, the number $N$ does not depend on $r.$

If $N=1$ then the prolongation exists: it is the argument of that only number.

Assume that the claim is proved for the case $<N.$ If not all numbers in $S(r_0)$ are identical,
then there exists a pair of reducing open sets $U_1$ and $U_2$ for $S$ in some left neighbourhood $(r_1,r_0],$
such that $U_1$ and $U_2$ contain non-empty parts of $S(r_0).$
By Lemma \ref{L: continuous liftings of reduced parts}, there exist continuous liftings of $S\big|_{U_1}$ and $S\big|_{U_2}$ on $[0,r_0).$
By induction assumption, continuous liftings of $S\big|_{U_1}$ and $S\big|_{U_2}$ can be prolonged continuously to $r_0.$
By Lemma \ref{L: sum is continuous}, the sum of those prolongations is continuous, and obviously, it is a prolongation of $\tilde S.$

If all numbers in $S(r_0)$ are identical and equal to $e^{i\theta_0},$ then there exists a reducing neighbourhood $U$ of $e^{i\theta_0}$
for $S$ on some left neighbourhood $W$ of $r_0.$ We set $\tilde S(r_0) := \set{\theta_0+2\pi n_1, \ldots, \theta_0 + 2\pi n_N}^*,$
where $\set{n_1,\ldots,n_N}^* = \set{[\theta/(2\pi)] \cdot 2 \pi \colon \theta \in S(r)}^*$ and $r$ is any number from $W.$


The proof is complete.
\end{proof}

Let
$$
  2\pi \mbZ := \set{2\pi k\colon k \in \mbZ} = p^{-1}(\set{1}).
$$
In the proof of existence of a continuous lifting of continuous functions $\colon \clX \to \euS_1(\mbT),$
we treat points $2\pi \mbZ$ as absorbing (sticky):
once a point reaches one of the points $2\pi k,$ $k \in \mbZ,$ it stays there forever.
So, we introduce the following formal definition.

\begin{defn} \label{D: standard lifting}
Let $\clX$ be a metric space, let $x_0 \in \clX,$ and let
$$S \colon \clX \to \euS_1(\mbT)$$ be a continuous function
such that $\supp S(x) \subset (-i,i)$ for all $x \in \clX.$

Let $\tilde S_0 \in \euS_1(\mbR)$ be such that $p \circ \tilde S_0 = S(x_0)$ and
$$
  \supp \tilde S_0 \subset \brs{-\frac \pi 2,\frac \pi 2} \cup 2\pi \mbZ.
$$

We define the \emph{standard lifting} of $S$ as that unique continuous function $\tilde S \colon \clX \to \euS_1(\mbR),$
such that
$$
  \supp \tilde S(x) \subset \brs{-\frac \pi 2,\frac \pi 2} \cup 2\pi \mbZ,
$$
$p\circ \tilde S(x) = S(x)$ for all $x \in \clX$ and $\tilde S(x_0) = \tilde S_0.$

The standard lifting map $S \mapsto \tilde S$ thus defined is obviously isometric.
\end{defn}

The following lemma is quite trivial.
\begin{lemma} The standard lifting exists and is unique.
\end{lemma}

The following intuitively obvious lemma (as many other lemmas in this section) belongs to the category of statements,
which are probably easier to reprove for oneself than to read a proof written by others.
\begin{lemma} \label{L: prolongation of right half-circle paths}
  Let $S \colon [0,r_0] \to \euS_1(\mbT)$ be a continuous path such that $\supp S(r) \subset (-i,i)$ for all $r \in [0,r_0].$
  Let $\tilde S \colon [0,r_0) \to \euS_1(\mbR)$ be a continuous lifting of $S$ on $[0,r_0).$
  Assume that all points $z \in S(r_0),$ which are not equal to $1,$ have the property: there exists a reducing for $S$ neighbourhood $U$ of $z$
  on a left neighbourhood $(r_1,r_0]$ of $r_0$ such that for all $r \in (r_1,r_0)$
  \begin{equation} \label{F: a formula}
    \supp \tilde S\big|_{p^{-1}(U)}(r) \subset \brs{-\frac\pi 2, \frac \pi 2}.
  \end{equation}
  Then the lifting $\tilde S$ can be continuously prolonged to $[0,r_0].$
\end{lemma} \margcom{OK \\ But read \\ again!}
\begin{proof}
For $r \in [0,r_0),$ let
$$
  \tilde S_1(r) := \set{\theta \in \tilde S(r) \colon \abs{\theta} \geq \frac \pi
  2}^*, \quad \tilde S_2(r) := \tilde S(r) - \tilde S_1(r).
$$
By Lemma \ref{L: continuity of restriction}, the functions $\tilde S_1$ and $\tilde S_2$ are continuous on $[0,r_0).$
The set $\tilde S_1(r)$ converges to a rigged subset of $2\pi \mbZ.$ Indeed, otherwise, there exists a sequence $r_1, r_2, \ldots$
converging to $r_0,$ and for every $n=1,2,\ldots$ there exist $\theta_n \in \tilde S_1(r_n),$ such that $\theta_n$ converge to some $\theta_0 \notin 2 \pi \mbZ.$
This implies that $z=e^{i\theta_0} \in S(r_0)$ and $z \neq 1;$ the point $z$ does not satisfy conditions of the lemma.

Further, since the continuous function $\tilde S_1$ can be continuously prolonged to $[0,r_0],$ the function $p \circ \tilde S_1$
is also continuous on $[0,r_0].$
By Theorem \ref{T: important estimate}, the function
$$
  S_2 = S - p \circ \tilde S_1
$$
is also continuous on $[0,r_0].$ Since the continuous lift $\tilde S_2$ of $S_2$ on $[0,r_0)$
takes values in $\brs{-\frac \pi 2,\frac \pi 2},$ it can be continuously prolonged to $[0,r_0]$ by the standard prolongation.
By Lemma \ref{L: sum is continuous}, it follows that $\tilde S = \tilde S_1 + \tilde S_2$ can also be continuously prolonged to $r_0.$

The proof is complete.
\end{proof}

\begin{lemma} \label{L: prolongation to [0,r]}
  Let $S \colon [0,1] \to \euS_1(\mbT)$ be a continuous path and let $r_0 \in (0,1].$
  Any continuous lifting $\tilde S \colon [0,r_0) \to \euS_1(\mbR)$ of $S$
  on $[0,r_0)$ can be continuously prolonged to $[0,r_0].$
\end{lemma} \margcom{OK \\ But read \\ again!}
\begin{proof}
Let $z \in S(r_0),$ $z \neq 1.$
Since $S$ is continuous, by Lemma \ref{L: X1 and X2} there exists a reducing neighbourhood $U_z$ of $z$ for $S$ in some left neighbourhood of $r_0,$
such that $U_z$ contains no other points of $S(r_0)$ except copies of $z$ itself.
By Lemma \ref{L: continuous liftings of reduced parts}, there exists a continuous lift $\tilde S\big|_{U_z}$ of the corresponding restriction.

Now, let $X_2$ be the set of all points $z$ of $S(r_0)$ not equal to $1,$ such that for all $r$ close enough to $r_0$
the rigged set $\tilde S\big|_{p^{-1}(U_z)}(r)$ contains at least one number
which does not belong to $\brs{-\frac \pi 2,\frac \pi 2}.$
The rigged set $X_2$ is finite, since otherwise the set $\tilde S(r)$ will not belong to $\euS_1(\mbR).$

It follows from Lemma \ref{L: X1 and X2} that for some small enough left neighbourhood $(r_1,r_0]$ of $r_0$
there exist a pair of reducing open subsets $U_1$ and $U_2$ of $\mbT$ for $S$ on $(r_1,r_0],$
such that $U_1 \supset X_1$ and $U_2 \supset X_2,$ where $X_1 := S(r_0) - X_2.$

It follows from Lemma \ref{L: continuous liftings of reduced parts} that there exist continuous liftings $\tilde S_1$ and $\tilde S_2$ of continuous
(by Lemma \ref{L: continuity of restriction}) functions $S\big|_{U_1}$ and $S\big|_{U_2}$ on $(r_1,r_0).$

It follows from Lemma \ref{L: prolongation of finite paths}, that $\tilde S_2$ admits continuous prolongation to $r_0.$
It also follows from Lemma \ref{L: prolongation of right half-circle paths}, that $\tilde S_1$ admits continuous prolongation to $r_0$ too.

Since $\tilde S = \tilde S_1 + \tilde S_2,$ it follows from
Lemma \ref{L: sum is continuous}, that $\tilde S$ admits continuous prolongation to $r_0.$

The proof is complete.
\end{proof}

\subsection{Completion of the proof}

\begin{lemma} \label{L: local lift exists} Let $\clX$ be a metric space, and let $x_0 \in \clX.$ Let $S \colon \clX \to \euS_1(\mbT)$
be a continuous function and let $\tilde S_0 \in \euS_1(\mbR)$ be such that $ p\circ \tilde S_0 = S(x_0).$
There exists a neighbourhood $W$ of $x_0,$ such that restriction of $S$ to $W$ admits a continuous lifting
$\tilde S \colon W \to \euS_1(\mbR)$ such that $\tilde S(x_0) = \tilde S_0.$
\end{lemma} \margcom{OK}
\begin{proof}
(A) Let $$A = \set{\theta \in \tilde S_0 \colon \abs{\theta} \geq \pi/2 \ \text{and} \ \theta \notin 2\pi \mbZ}^*$$
and let
$$
  Y_2 = A + \set{\theta \in \tilde S_0 \colon \theta - 2\pi k \in A \ \text{for some} \ k \in \mbZ}^*.
$$
Let $Y_1 = \tilde S_0 - Y_2,$ $X_j = p(Y_j).$
Clearly, $S(x_0) = X_1 \sqcup X_2,$ and $X_2$ is finite.

%
%

By Lemma \ref{L: X1 and X2}, there exists a reducing pair of open neighbourhoods $U_1$ and $U_2$ of $X_1$ and $X_2$ respectively
for $S$ in some right neighbourhood $W = [r_0,r_1)$ of $r_0.$ The neighbourhood $U_1$ can be chosen so that $U_1 \subset (-i,i).$
By Lemma \ref{L: continuity of restriction}, the corresponding restrictions $S\big|_{U_1}$ and $S\big|_{U_2}$ are continuous on $W$
and $S = S\big|_{U_1} + S\big|_{U_2}.$

Since
$$
  \supp (\tilde S_0 \cap p^{-1}U_1) \subset \brs{-\frac \pi 2,\frac \pi 2} \cup 2\pi \mbZ,
$$
the function $S\big|_{U_1}$ has a continuous lifting $\tilde S\big|_{U_1}$ to $W$ by standard prolongation,
such that
$$
  \tilde S \big|_{U_1} (x_0) = \tilde S_0 \cap p^{-1} U_1.
$$

(B) Claim: the function $S\big|_{U_2}$ has a continuous lifting $\tilde S\big|_{U_2}$ to $W,$
such that
$$
  \tilde S \big|_{U_2} (x_0) = \tilde S_0 \cap p^{-1} U_2.
$$
Proof of (B). Since the rigged set $X_2$ is finite, we obviously can assume that it consists of only one point and its copies.
This one point case follows from Lemma \ref{L: one-point lift}. 

Combining (A) and (B), it follows from Lemma \ref{L: sum is continuous} that the function $\tilde S:=\tilde S\big|_{U_1} + \tilde S\big|_{U_2}$
gives continuous lifting of $S$ to $W,$ such that $\tilde S(x_0) = \tilde S_0.$
\end{proof}

\begin{thm} \label{T: Hurevich bundle} The triple $\brs{\euS_1(\mbR),\euS_1(\mbT),p}$ is a Hurevich bundle;
that is, if $\tilde S_0 \in \euS_1(\mbR),$ if $S \colon [0,1] \to \euS_1(\mbT)$ is continuous and if $p(\tilde S_0) = S(0),$
then there exists a continuous path $\tilde S \colon [0,1] \to \euS_1(\mbR),$ such that $\tilde S(0) = \tilde S_0$
and $S = p \circ \tilde S.$
\end{thm} \margcom{OK}
\begin{proof}
Let $A$ be the set of all $r_0 \in [0,1],$ such that the restriction of the path $S$ to $[0,r_0]$ has a continuous lifting.
Clearly, $0 \in A,$ so that $A \neq \emptyset.$

It follows from Lemma \ref{L: local lift exists} that if $r_0 \in A$ and $\tilde S$ is a lifting of $S$
on $[0,r_0],$ then there exists a right neighbourhood $W = [r_0,r_1)$ of $r_0,$ such that $\tilde S$
can be continuously prolonged to $W.$ Hence, the set $A$ is open.
It follows from Lemma \ref{L: prolongation to [0,r]} that $A$ is closed. Consequently,
$A$ is a non-empty closed and open subset of $[0,1].$ It follows that $A = [0,1].$
\end{proof}

The path $\tilde S,$ existence of which is proved in Theorem \ref{T: Hurevich bundle}, is called a lifting of the path $S.$
This lifting is not unique, in general. In this regard, it is desirable to study some relationship between different liftings.


\bigskip

Now we are in position to prove Theorem \ref{T: Selection Thm}.

{\it Proof of Theorem \ref{T: Selection Thm}.} By Theorem \ref{T: Hurevich bundle},
there exists a continuous lift $\tilde S \colon [0,1] \to \euS_1(\mbR)$ of the path $S.$
By Lemma \ref{L: continuous sections of path S}, there exists a sequence of continuous
functions $\theta_1, \theta_2, \ldots \colon [0,1] \to \mbR$
such that $\set{\theta_1(r),\theta_2(r),\ldots}^* = \tilde S(r).$ Obviously, the functions
$z_j(r) := e^{2\pi i \theta_j(r)}$ are continuous, and $\set{z_1(r),z_2(r),\ldots}^* = S(r).$
$\Box$

\section{$\mu$-invariant of a continuous path in $\euS_1(\mbT)$}

\subsection{Definition of $\mu$-invariant}
\subsubsection{The number $[\theta; \theta_1,\theta_2]$}
Let $\theta \in (0,2\pi).$ For any $\theta_1, \theta_2 \in \mbR,$ such that $\theta_1 < \theta_2,$
we define
$$
  [\theta; \theta_1, \theta_2] = \frac 12 \brs{\#\set{k \in \mbZ \colon \theta_1 < \theta+2\pi k < \theta_2} + \#\set{k \in \mbZ \colon \theta_1 \leq \theta+2\pi k \leq \theta_2}}.
$$
This number is equal to the number of times the point $e^{it}$ cross $e^{i\theta}$
in anti-clockwise direction as $t$ moves from $\theta_1$ to $\theta_2.$
If $\theta_2 < \theta_1,$ we let
$$
  [\theta; \theta_1, \theta_2] := - [\theta; \theta_2, \theta_1].
$$
Clearly, for any three numbers $\theta_1, \theta_2, \theta_2$
\begin{equation} \label{F: additivity of [theta;...]}
  [\theta; \theta_1, \theta_3] := [\theta; \theta_1, \theta_2] + [\theta; \theta_2, \theta_3].
\end{equation}

\subsubsection{Definition of $\mu$-invariant}
If we have a continuous path in $\euS_1(\mbT),$ it is desirable to know how many points in total cross a particular
point $\theta$ on the unit circle in counterclockwise direction and how many points cross that point in clockwise direction.
Actually, as it is easy to see, only difference of the above two numbers can be correctly defined. This number can be considered
as spectral flow in the case when the path in $\euS_1(\mbT)$ represents the changing spectra of a path of unitary operators of the class
$1+\clL_1(\hilb).$ This spectral flow was called $\mu$-invariant by A.\,Pushnitski who introduced this notion in \cite{Pu01FA}.

The definition of the $\mu$-invariant which follows is based on
Selection Theorem \ref{T: Selection Thm} and as such it differs from the one given in \cite{Pu01FA}.

\begin{defn} Let $a<b$ be two real numbers. Let $S \colon [a,b] \to \euS_1(\mbT)$ be a continuous path
and let $\tilde S \colon [a,b] \to \euS_1(\mbR)$ be any continuous lifting of the path $S.$
Let $\theta_1(\cdot), \theta_2(\cdot), \ldots$ be a continuous enumeration of $\tilde S.$ The $\mu-invariant$ of the path $S$
is a function
$$
  \mu(\theta; a,b) \colon (0,2\pi) \to \mbZ \cup \frac 12 \mbZ,
$$
defined by the formula
$$
  \mu(\theta; a,b) = \sum_{j=1}^\infty [\theta; \theta_j(a),\theta_j(b)].
$$
\end{defn}
This definition assumes that if a point of the path $S(r)$ arrives to $e^{i\theta}$ (in anticlockwise direction) it adds $\frac 12$ to the $\mu$-invariant and when
the point leaves $e^{i\theta}$ it adds another $\frac 12;$ as a result, when a point crosses $e^{i\theta}$ it adds $1$ to the $\mu$-invariant.

Since the rigged set $\set{\theta_j(r)}_{j=1}^\infty$ belongs to $\euS_1(\mbR),$
the above sum is finite for every $\theta \in (0,2\pi).$

Note that for any fixed $r$ the $\mu$-invariant is a locally constant function of $\theta,$ whose jumps occur at
points of $S(r).$ At discontinuity points the $\mu$-invariant takes half-integer values (which include integers too),
at continuity points the $\mu$-invariant is integer-valued. So, the $\mu$-invariant is essentially integer-valued.
Though $\mu$-invariant can take half-integer values as well, we shall usually ignore this.

Obviously, one needs to prove correctness of the definition of the $\mu$-invariant given above;
that is, to show that the definition does not depend on the choice of continuous enumeration $\theta_1, \theta_2, \ldots.$
The proof of the correctness which follows is a routine and straightforward check.

\begin{prop} \label{P: mu is fine} The $\mu$-invariant is correctly defined; that is, it does not depend on the choice of lifting $\tilde S$
and it does not depend on the choice of enumeration $\theta_1, \theta_2, \ldots$ of $\tilde S.$
\end{prop}
\begin{proof}
It is obvious that the $\mu$-invariant does not depend on rearrangement of functions in a given enumeration $\set{\theta_j(\cdot)}.$

(A) Let $A$ be the set of all $r_0 \in [0,1]$ such that for the restriction of the path $S(\cdot)$
to the interval $[0,r_0]$ definition of the $\mu$-invariant is correct; that is, that is does not depend on the choice
of the continuous enumeration $\set{\theta_j(r)}.$

We shall prove that $A = [0,1].$ Obviously, $0 \in A,$ so that $A \neq \emptyset.$
%
%
From now on we assume that $\theta \in (0,2\pi)$ is fixed.

(B) Here we show that $A$ is open.

So, let $r_0 \in A,$ that is, for $r \leq r_0$ definition of $\mu$-invariant at $\theta$ is correct for all $r \leq r_0.$
If $e^{i\theta} \notin S(r_0),$ then there exists a neighbourhood $U$ of $e^{i\theta}$
such that the distance between $U$ and $S(r_0)$ is positive, and, in particular, $U$ contains
no elements of $S(r_0).$ Since $S(\cdot)$ is continuous in $\euS_1(\mbT),$ it follows that for all $r$ close enough to $r_0$
the support of the set $S(r)$ also does not intersect with $U.$ It follows that none of the points of $S(r)$ cross $e^{i\theta};$
that is, values of all functions $\theta_j(\cdot)$ are not equal to any of the numbers $\theta +2\pi k,$ $k \in \mbZ.$
It follows that the value of $\mu$-invariant does not change for all $r$ close enough to $r_0,$ independently of the enumeration.

Now assume that $e^{i\theta} \in S(r_0).$ Then there exists a neighbourhood $U$ of $e^{i\theta}$ which contains only (copies of) one point from $S(r_0),$
and such that the distance between $U$ and $S(r_0) \setminus \set{e^{i\theta}}$ is positive. Let (without loss of generality) $\theta'_1(\cdot), \ldots, \theta'_N(\cdot)$
be those and only functions of the first enumeration, for which $e^{i\theta'_j(r_0)} = e^{i\theta},$ and let $\theta''_1(\cdot), \ldots, \theta''_N(\cdot)$
be the corresponding functions of the second enumeration with $e^{i\theta''_j(r_0)} = e^{i\theta}$ (clearly, the number $N$ of the functions is the same).
It follows that there exists a neighbourhood $W$ of $r_0,$
such that for all $r \in W$ the set $U$ contains only points $e^{i\theta'_1(r)}, \ldots, e^{i\theta'_N(r)}$ of $S(r);$ these points coincide with
$e^{i\theta''_1(r)}, \ldots, e^{i\theta''_N(r)}.$ For every $r \in W$ the number of numbers from $\theta'_1(\cdot), \ldots, \theta'_N(\cdot)$
which are larger than $\theta$ minus the number of numbers from $\theta'_1(\cdot), \ldots, \theta'_N(\cdot)$
which are less than $\theta$ represent the change of the $\mu$-invariant. Since this difference is clearly the same for the functions
$\theta''_1(\cdot), \ldots, \theta''_N(\cdot),$ the value of the $\mu$-invariant does not depend on enumeration for all $r \in W.$

Proof of (B) is complete.

(C) Let $r_0 \in (0,1].$ Assume that for all $r \in [0,r_0)$ the definition of the $\mu$-invariant is correct, that is $[0,r_0) \subset A.$
We show that $r_0 \in A,$ and this will complete the proof.

Again, we consider two cases: (1) $e^{i\theta} \notin S(r_0)$ and (2) $e^{i\theta} \in S(r_0).$

First case: $e^{i\theta} \notin S(r_0).$ In this case there exists a neighbourhood $U$ of $e^{i\theta}$
such that the distance between $U$ and  $S(r_0)$ is positive. It follows that there exists a neighbourhood $W$ of $r_0,$
such that for all $r \in W$ the set $U$ does not intersect with $S(r).$ This clearly implies that the $\mu$-invariant is the same
for $r_0$ and any $r \in W$ for both enumerations.

Second case: $e^{i\theta} \in S(r_0).$ Let $N$ be the multiplicity of $e^{i\theta}$ in $S(r_0).$
Let, as in part (B), $\theta'_1(\cdot), \ldots, \theta'_N(\cdot)$
be those and only functions of the first enumeration, for which $e^{i\theta'_j(r_0)} = e^{i\theta},$ and let $\theta''_1(\cdot), \ldots, \theta''_N(\cdot)$
be the corresponding functions of the second enumeration with $e^{i\theta''_j(r_0)} = e^{i\theta}$ (clearly, the number $N$ of the functions is the same).

In this case there exists a neighbourhood $U$ of $e^{i\theta}$
such that the distance between $U$ and  $S(r_0) \setminus \set{e^{i\theta}}$ is positive. It follows that there
exists a neighbourhood $W$ of $r_0$ such that for all $r \in W$ the set $U$ contains only points $e^{i\theta'_1(r)}, \ldots, e^{i\theta'_N(r)},$
which coincide with points $e^{i\theta''_1(r)}, \ldots, e^{i\theta''_N(r)}.$ By assumption, for any fixed $r < r_0$ the $\mu$-invariant have the same value
for both enumerations. The change of the $\mu$-invariant on the interval $[r,r_0]$ is represented by
$$
  \# \set{j \in 1\ldots N: \theta'_j(r) > \theta'_j(r_0)} - \# \set{j \in 1\ldots N: \theta'_j(r) < \theta'_j(r_0)}.
$$
Since this number is clearly the same for both enumerations $\set{\theta'_j}$ and $\set{\theta''_j},$ it follows that $r_0 \in A.$

The proof is complete.
\end{proof}

%
%
%

\begin{prop} \label{P: mu is additive} Let $S \colon [a,b] \to \euS_1(\mbT)$ be a continuous
path. The $\mu$-invariant of $S$ is additive in the sense that for any $a,c,b \in \mbR$
$$
  \mu(\theta; a,b) = \mu(\theta; a,c) + \mu(\theta; c,b).
$$
\end{prop}
\begin{proof} Directly follows from the definition of $\mu$-invariant and from (\ref{F: additivity of [theta;...]}).
\end{proof}

As Corollary \ref{C: d = rho1} shows, definitions of the $\mu$-invariant given here and in \cite{Pu01FA} coincide.
At the same time, it is often more convenient to work with separate continuous eigenvalue-functions. While definition of $\mu$-invariant,
given here, might be a bit lengthy and technical, once introduced, it clarifies and simplifies many things.

\subsection{Homotopy invariance}
An important property of the $\mu$-invariant is its homotopy invariance.
Proof of homotopy invariance is also standard, see e.g. proof of homotopy invariance of spectral flow in \cite{Ph96CMB,Ph97FIC}.

%
%
%
%
%

\begin{thm} \label{T: mu-invariant is homotopically invariant}
The $\mu$-invariant is homotopically invariant. That is, if $F \colon [a,b] \times [0,1] \to \euS_1(\mbT)$
is a continuous function, such that $F(a,t) = 1$ for all $t \in [0,1]$ and $F(b,t) = F_1 = \const,$ then the paths $S_a(\cdot) = F(\cdot,0)$
$S_b(\cdot) = F(\cdot,1)$ have the same $\mu$-invariants.
\end{thm}
\begin{proof}
%
%
%
%

Let $\theta \in (0,2\pi)$ be fixed.

(A) Let $x_0 \in [a,b] \times [0,1].$ There exists a neighbourhood $W$ of $x_0$ such that for any two points $x_1, x_2 \in W$
and any two paths $I_1$ and $I_2$ in $U$ which begin at $x_1$ and end at $x_2,$ the equality
$$
  \mu(\theta; F\big|_{I_1}) = \mu(\theta; F\big|_{I_2})
$$
holds.

Proof of (A). 1 case: $\theta \notin F(x_0).$ In this case there exists a neighbourhood $W$ of $x_0$ and a neighbourhood $U$ of $\theta$
which is reducing for $F$ on $W;$ that is, for any $x \in W$
$$
  \supp F(x) \cap U = \emptyset.
$$
Plainly, this implies that $\mu(\theta; F\big|_{I_1}) = \mu(\theta; F\big|_{I_2}).$

2 case: $\theta \in F(x_0).$ In this case there exists a neighbourhood $W$ of $x_0$ and a neighbourhood $U$ of $\theta$
which is reducing for $F$ on $W$ and such that $\supp F(x_0) \cap U$ contains only $\theta.$
This restriction, defined on $W,$ admits a (standard) continuous lifting $\tilde F\big|_{U}$ which admits a continuous enumeration
$\theta_1(\cdot), \ldots, \theta_N(\cdot),$ where $N$ is the multiplicity of $x_0.$
(by Lemma \ref{L: continuous sections of path S}). Restrictions of the functions
$\theta_1(\cdot), \ldots, \theta_N(\cdot)$ to $I_1$ and $I_2$ give continuous enumerations $\theta'_j(\cdot)$ and $\theta''_j(\cdot)$ of $F\big|_{I_1}$
and $F\big|_{I_2}.$ Clearly,
$$
  [\theta; \theta'_j(x_1), \theta'_j(x_2)] = [\theta; \theta''_j(x_1), \theta''_j(x_2)].
$$
It follows that $\mu(\theta; F\big|_{I_1}) = \mu(\theta; F\big|_{I_2}).$

(B) The rest of the proof is standard: it uses (A) and additivity property (Proposition \ref{P: mu is additive}) of the $\mu$-invariant.
\end{proof}

\subsection{A property of $\mu$}

\begin{prop} \label{P: mu is constant, if ...}
If $S$ is a continuous path $[0,1] \to \euS_1(\mbT)$
which begins and ends at $1,$
then the $\mu$-invariant of $S$ is a constant function.
\end{prop}
\begin{proof} $\mu$-invariant is a step-function on $(0,2\pi);$ it follows from the definition
of the $\mu$-invariant that its jumps occur at points which belong either
to $S(0)$ or $S(1).$ The claim follows.
\end{proof}

\begin{cor} \label{C: mu-inv. = const}
If $S$ and $T$ are two continuous paths $[0,1] \to \euS_1(\mbT)$
such that $S(0) = T(0) = {\bf 1}$ and $S(1) = T(1),$
then the difference
$$
  \mu(\theta; S) - \mu(\theta; T)
$$
is constant (does not depend on $\theta$).
\end{cor}
\begin{proof} This follows from additivity of the $\mu$-invariant (Proposition \ref{P: mu is additive})
and Proposition~\ref{P: mu is constant, if ...}.
\end{proof}

\begin{lemma} \label{L: (f(j)=-1) finite}
  Let $S \colon [0,1] \to \euS_1(\mbT)$ be a continuous function such that $S(0) = {\bf 1}.$
  For any continuous enumeration $\set{z_1, z_2, \ldots}^*$ of $S,$ the set
  $$
    \set{j \in \mbN \colon z_j(r) = -1 \ \text{for some} \ r \in [0,1]}
  $$
  is finite.
\end{lemma}
\begin{proof} Let $\tilde S \colon [0,1] \to \euS_1(\mbR)$ be a continuous lifting of $S$
(such a lifting exists by Theorem \ref{T: Hurevich bundle})
and let $\set{\theta_1,\theta_2,\ldots}^*$ be a continuous enumeration of $\tilde S$
(it exists by Lemma \ref{L: continuous sections of path S}).
If we assume the contrary to the claim of the lemma, then there exists a sequence $r_1,r_2,\ldots \in [0,1]$
and a sequence $n_1, n_2, \ldots$ of indices such that $\abs{\theta_{n_j}(r_j)} \geq \pi.$
But this contradicts to Lemma \ref{L: sum theta converges uniformly}.
\end{proof}

\begin{prop} \label{P: homotopic imply the same mu-invariant} Let $S$ and $T$ be two continuous paths $[0,1] \to \euS_1(\mbT)$
such that $S(0) = T(0) = S(1) = T(1) = {\bf 1}.$ If $S$ and $T$ have the same $\mu$-invariants (which are necessarily constant, by Corollary \ref{C: mu-inv. = const}),
then $S$ and $T$ are homotopic.
\end{prop}
\begin{proof} (A) Let $\mu(\theta) = N.$ W.l.o.g, we assume that $N \geq 0.$
We show that $S$ is homotopic to the continuous path
\begin{equation} \label{F: set(exp(2pi x))}
  [0,1] \ni r \mapsto \set{e^{2\pi i r}, e^{2\pi i r}, \ldots, e^{2\pi i r}}^*,
\end{equation}
where the number $e^{2\pi i r}$ appears $N$ times.
This will imply that $T$ is also homotopic to $1;$ consequently $S$ and $T$ are homotopic.

(B) By Theorem \ref{T: Selection Thm}, there exists a continuous enumeration $\set{z_1, z_2, \ldots}^*$ of $S.$
Without loss we can assume that functions $z_1,z_2,\ldots$ have the form $e^{i\theta_j(r)}$ where functions $\theta_j$ take their values either in $[0,2\pi]$
or in $[-2\pi,0],$ and that the values $-2\pi,0,2\pi$ are attained only at end-points of $[0,1].$ \margdetails

We split the set $S(r)$ into two parts $S_1(r)$ and $S_2(r),$ where $S_1(r)$
consists of those $z_j(r),$ for which the function $z_j(\cdot)$ takes value $-1,$
and $S_2(r) := S(r) - S_1(r).$ By Corollary \ref{C: subset of enumeration is continuous},
functions $S_1$ and $S_2$ are continuous. By Lemma \ref{L: (f(j)=-1) finite}, the rigged set $S_1(r)$
is finite.

Let $S_2=\set{g_1, g_2, \ldots}^*$ and $S_1=\set{h_1,h_2,\ldots,h_N}^*.$

(C) Since continuous arguments $\theta''_j$ of functions $g_j$, which form $S_2,$ do not take values $\pm \pi,$
so that $\theta_j''(0) = \theta_j''(1) = 0,$
we get a homotopy of $S_2$ with $1,$ if we let $F_2(r,t) = \set{e^{it\theta_j''(r)}, j=1,2,\ldots}^*.$
It follows from Lemma \ref{L: cont-s f-ns imply cont-s set f-n} that the function $F_2 \colon [0,1]^2 \to \euS_1(\mbT)$
is continuous, so $F_2$ is indeed a homotopy.

(D) Functions $h_j,$ which form $S_1,$ can be represented in the form $h_j(r) = e^{i\theta'_j(r)},$ where the argument $\theta'_j(\cdot)$ is a continuous non-negative function
with values in $[0,2\pi],$ of one of three types: (1) \ $\theta'_j(0) = \theta'_j(1) = 0,$ \ (2) \ $\theta'_j(0) = 0, $ $\theta'_j(1) = 2\pi,$  and (3) \ $\theta'_j(0) = 2\pi, $ $\theta'_j(1) = 0.$

If the rigged set $\set{\theta'_1, \theta'_2, \ldots, \theta'_N}^*$ contains a pair functions $\theta'_k$ and $\theta'_m,$ of types (2) and (3),
then we obviously can replace this pair by two functions of the first type. In the end we get some number of functions
of the first type and $N$ functions of the second or third. We can assume that the functions are of the second type.
Functions of the first type form a continuous subfunction of $S_1$ which is homotopic to $1$ just as in (C).

Each function $h_j$ of the second (or third) type is obviously homotopic to the function $e^{2\pi ir}.$
By Lemma \ref{L: cont-s f-ns imply cont-s set f-n}, these homotopies define a (continuous) homotopy $F_1(\cdot,\cdot)$ of $S_1$ to the set function
$$
  \set{e^{2\pi i r}, e^{2\pi i r}, \ldots, e^{2\pi i r}}^*,
$$
By Lemma \ref{L: sum is continuous}, the function $F:=F_1 + F_2$ is continuous.
Clearly, $F$ is a homotopy of $S$ and the path (\ref{F: set(exp(2pi x))}).

The proof is complete.
\end{proof}

Since the $\mu$-invariant of the path (\ref{F: set(exp(2pi x))}) is obviously constant and is equal to $N,$
Proposition \ref{P: homotopic imply the same mu-invariant} immediately implies
\begin{cor} \label{C: fundamental group of euS(T)} The fundamental group $\pi_1\brs{\euS_1(\mbT)}$ of the space $\euS_1(\mbT)$ is equal to $\mbZ.$
\end{cor}

%
%
%
%
%

\subsection{Another property of $\mu$}
The following simple equality will be used twice.
\begin{lemma} \label{L: SS3 int mu = sum theta j}
Let $S \colon [0,1] \to \euS_1(\mbT)$ be a continuous path, such that $S(0) = {\bf 1}.$
Let $\tilde S \colon [0,1] \to \euS_1(\mbR)$ be its continuous lift such that $\tilde S(0) = {\bf 0},$
and let $\set{\theta_1(\cdot),\theta_2(\cdot),\ldots}^*$
be its continuous enumeration.
The $\mu$-invariant of the path $S$ is a summable function and it satisfies the equality
the equality
  \begin{equation} \label{F: int of mu = sum of theta}
    \int_0^{2\pi} \mu(\theta; S)\,d\theta = \sum_{j=1}^\infty \theta_j(1).
  \end{equation}
\end{lemma}

\begin{proof}
Let $\theta_j$ be the number from $[0,2\pi)$ which differs from $\theta_j(1)$ by a multiple of $2\pi.$
If $\theta_j(r)$ makes $k \in \mbZ$ windings around the unit circle
in anticlockwise direction as $r$ changes from $0$ to $1,$
then
\begin{multline*}
  \int_0^{2\pi} \SqBrs{\frac {\theta-\theta_j(1)}{2\pi}}\,d\theta = \int_0^{\theta_j} + \int_{\theta_j}^{2\pi}
   \\ = - \theta_j\cdot(k+1) - (2\pi - \theta_j)k = -2\pi k - \theta_j = -\theta_j(1).
\end{multline*}
It follows that for all $j=1,2,\ldots$
$$
  \abs{\theta_j(1)} =
  \abs{\int_0^{2\pi} \SqBrs{\frac {\theta-\theta_j(1)}{2\pi}}\,d\theta}
    = \int_0^{2\pi} \abs{\SqBrs{\frac {\theta-\theta_j(1)}{2\pi}}}\,d\theta.
$$
Since the series
$\sum_{j=1}^\infty \theta_j(1)$
is absolutely convergent, it follows that the series
$$
  \sum_{j=1}^\infty \int_0^{2\pi} \abs{\SqBrs{\frac {\theta-\theta_j(1)}{2\pi}}}\,d\theta
$$
is convergent. It follows that the function
$$
  (0,2\pi) \ni \theta \mapsto \sum_{j=1}^\infty \abs{\SqBrs{\frac {\theta-\theta_j(1)}{2\pi}}}
$$
is summable.
It follows that the function $(0,2\pi) \ni \theta \mapsto \mu(\theta; S)$ is also summable and
\begin{equation*}
 \begin{split}
  \int_0^{2\pi} \mu(\theta; S)\,d\theta
    & = - \int_0^{2\pi} \sum_{j=1}^\infty \SqBrs{\frac {\theta-\theta_j(1)}{2\pi}} \,d\theta
    \\ & = - \sum_{j=1}^\infty \int_0^{2\pi} \SqBrs{\frac {\theta-\theta_j(1)}{2\pi}} \,d\theta
    \\ & = \sum_{j=1}^\infty \theta_j(1).
 \end{split}
\end{equation*}
The proof is complete.
\end{proof}

\section{$\mu$-invariant for paths of unitary operators}
\newcommand{\spec}{\mathrm{spec}}
\subsection{The mapping $\spec\colon \clU_1 \to \euS_1$\!}
Let $p \in [1,\infty].$
We denote by $\clU_p(\hilb)$ the group of all unitary operators of the form $1+A,$ where $A \in \clL_p(\hilb),$
on a Hilbert space $\hilb.$ 

The group $\clU_p(\hilb)$ is endowed with $\clL_p(\hilb)$ topology:
a net $U_\alpha$ from $\clU_p(\hilb)$ converges to $U \in \clU_p(\hilb),$ if $U_\alpha - U$ converges to $0$
in $\clL_p(\hilb).$ In this paper $p$ will be equal to $1.$

We shall consider continuous paths of operators $U \colon [0,1] \to \clU_p(\hilb).$

Eigenvalues $e^{i\theta_j}$ of an operator $U \in \clU_p(\hilb)$ form a subset of the unit circle $\mbT \subset \mbC$
with only one possible accumulation point $1.$
In an obvious way, spectrum $\spec(U)$ of an operator $U \in \clU_1(\hilb)$ is a rigged set from $\euS_\infty(\mbT).$
Since $\norm{U-1}_1 = d(\spec(U),1),$ it follows that spectrum defines a mapping from $\clU_1(\hilb)$
to $\euS_1(\mbT):$
$$
  \spec \colon \clU_1(\hilb) \to \euS_1(\mbT).
$$
In \cite{Pu01FA} this mapping is denoted by $\eta_1.$

\begin{lemma} \label{L: sum abs(f(j),A f(j))<...} Let $\hilb$ be a Hilbert space (finite or infinite dimensional) and let $A$ be a trace class
operator on $\hilb.$ Then for any orthonormal basis $(\psi_j)$ of $\hilb$
$$
  \sum_{j=1}^\infty \abs{\scal{\psi_j}{A\psi_j}} \leq \norm{A}_1.
$$
\end{lemma}
\begin{proof} There exists a diagonal (in the basis $(\psi_j)$) unitary operator $U$ such that
the diagonal entries of the operator $AU$ (in the same basis) are non-negative.
It follows that
$$
  \sum_{j=1}^\infty \abs{\scal{\psi_j}{A\psi_j}} = \Tr(AU) \leq \norm{AU}_1 = \norm{A}_1.
$$
\end{proof}

The formula proved in the following lemma is well-known (for self-adjoint operators) in quantum mechanics (see e.g. \cite{LL3}).
\begin{lemma} \label{L: velocity of eigenvalue} If $\lambda_j(r_0)$ is an isolated eigenvalue of a real-analytic path of unitary operators $U(r)$ at $r = r_0$
corresponding to an eigenvector $\psi_j(r_0),$ then
$$
  \frac {d\lambda_j(r)}{dr}  \big|_{r=r_0} = \scal{\psi_j(r_0)}{\dot U_r \big|_{r=r_0} \psi_j(r_0)}.
$$
\end{lemma}
\begin{proof} 
Using the fact that $\bar \lambda_j(r)$ is an eigenvalue of $U^{-1}(r)=U^*(r),$ one obtains
(recall that $\scal{\cdot}{\cdot}$ is anti-linear with respect to the first argument)
\begin{equation*}
  \begin{split}
    \frac {d\lambda_j(r)}{dr}
    & = \frac {d}{dr} \scal{\psi_j(r)}{U(r)\psi_j(r)} \\
    & = \scal{\psi_j'(r)}{U(r)\psi_j(r)} + \scal{\psi_j(r)}{U(r)\psi_j'(r)} + \scal{\psi_j(r)}{U'(r)\psi_j(r)} \\
    & = \lambda_j(r)\scal{\psi_j'(r)}{\psi_j(r)} + \lambda_j(r)\scal{\psi_j(r)}{\psi_j'(r)} + \scal{\psi_j(r)}{U'(r)\psi_j(r)}.
  \end{split}
\end{equation*}
Since $\scal{\psi_j'(r)}{\psi_j(r)} + \scal{\psi_j(r)}{\psi_j'(r)} = \brs{\scal{\psi_j(r)}{\psi_j(r)}}' = 0,$ the claim follows.
\end{proof}

\begin{lemma} \label{L: velocity of eigenvalues}
  Let $\set{U(r), r \in [0,1]}$ be an analytic path of unitary operators in finite-dimensional Hilbert space. 
  Let $\lambda_1(r), \ldots, \lambda_N(r)$ be eigenvalues of $U(r),$  which are analytic functions of $r.$
  Then
  $$
    \sum_{j=1}^N \abs{\frac{d \lambda_j(r)}{dr}} \leq \norm{\frac{dU(r)}{dr}}_1.
  $$
\end{lemma}
\begin{proof}
This directly follows from Lemmas \ref{L: velocity of eigenvalue}, \ref{L: sum abs(f(j),A f(j))<...}
and the fact that analytic vectors $\psi_1(r),\ldots,\psi_N(r)$ form an orthonormal basis.
\end{proof}

\begin{lemma} \label{L: spec distance estimate} Let $a<b.$ If $U \colon [a,b] \to \clU_1(\mbC^N)$ is a real-analytic path, then
\begin{equation} \label{F: spec distance estimate}
  d(\spec(U(a)),\spec(U(b))) \leq \max_{r \in [a,b]} \norm{\frac {dU(r)}{dr}}_1 (b-a).
\end{equation}
\end{lemma}
\begin{proof}
Plainly,
$$
  d(\spec(U(a)),\spec(U(b))) \leq \sum_{j=1}^N d(\lambda_j(U(a)),\lambda_j(U(b))),
$$
where the functions $\lambda_j(U(r))$ can be chosen to be analytic (since the path $U(\cdot)$ is unitary and analytic).
So, the mean value theorem and Lemma \ref{L: velocity of eigenvalues} complete the proof.
\end{proof}

We introduce metric in $\clU_1(\hilb)$ by the formula
$$
  d_1(U_0,U_1) = \inf \max_{r \in [a,b]} \norm{\frac {dU(r)}{dr}}_1 (b-a)
$$
where the infimum is taken over all piecewise analytic paths $U \colon [a,b] \to \clU_1(\hilb)$ such that $U(a) = U_0$
and $U(b) = U_1.$ Plainly, $\brs{\clU_1(\hilb),d_1}$ is a metric space.

\begin{lemma} \label{L: d1(UU1,UU2)=d1(U1,U2)} For any $U_1, U_2, U \in 1 + \clL_1(\hilb),$ the equality
$$
  d_1(UU_1, UU_2) = d_1(U_1, U_2).
$$
holds.
\end{lemma}
\begin{proof} If $U(r)$ is a piecewise analytic path connecting $U_1$ and $U_2$ with trace-class derivative $ \leq 1,$
then $U \cdot U(r)$ is a piecewise analytic path connecting $UU_1$ and $UU_2$ with trace-class derivative $ \leq 1.$
It follows that $d_1(UU_1, UU_2) \leq d_1(U_1, U_2).$
Using this, we obtain
$$
  d_1(U_1, U_2) = d_1(U^{-1}U U_1, U^{-1}U U_2) \leq d_1(U U_1, U U_2).
$$
\end{proof}

\begin{lemma} \label{L: d1 and norm1 are equiv} Metrics $\norm{U_1-U_0}_1$ and $d_1(U_1,U_0)$ are equivalent.
\end{lemma}
\begin{proof} (A) Claim:
$$
  \norm{U_1-U_0}_1 \leq d_1(U_1,U_0).
$$

Proof. Let $U \colon [0,1] \to \clU_1(\hilb)$ be any differentiable path,
connecting $U_0$ and $U_1.$
Since
$$
  U_1 - U_0 = \int_0^1 U'(r)\,dr,
$$
(where the integral is taken in $\clL_1$-norm),
it follows that
$$
  \norm{U_1-U_0}_1 \leq \int_0^1 \norm{U'(r)}_1\,dr.
$$
It follows that $\norm{U_1-U_0}_1 \leq d_1(U_1,U_0).$

(B) Claim: there exists an absolute constant $C>0$ such that
$$
  d_1(U_1,U_2) \leq C \norm{U_1-U_2}_1.
$$

Proof. By Lemma \ref{L: d1(UU1,UU2)=d1(U1,U2)}, $d_1(U_1,U_2) = d_1(1,U),$ where $U=U_2U_1^{-1}.$
Since also $\norm{U_2-U_1}_1 = \norm{U_2U_1^{-1}-1}_1,$ we can assume that $U_1 = 1$ and $U_2 = U.$

Let $U = e^{iH},$ where $\norm{H} \leq \pi$ and $H \in \clL_1(\hilb).$
The path $U(r) = e^{irH}$ connects $1$ and $U.$
Further,
$$
  \frac {d U(r)} {dr} = i H e^{irH},
$$
so that
$$
  d_1(1,U) \leq \norm{H}_1.
$$
So, it is enough to show that
$$
  \norm{H}_1 \leq C \cdot \norm{e^{iH}-1}_1.
$$
If $\lambda_1, \lambda_2, \ldots$ are eigenvalues (counting multiplicities) of $H,$ then, by the formula $\norm{H}_1 = \sum_{j=1}^\infty \abs{\lambda_j},$
it is enough to show that
$$
  \abs{\lambda_j} \leq C \cdot \abs{e^{i\lambda_j}-1},
$$
with $C$ which does not depend on $j.$
Since the function $f(x) := \frac x{\abs{e^{ix}-1}}$ is (can be made) continuous on $[0,\pi],$ the constant $C$ exists.
\end{proof}
This lemma shows that we can work with either of the two metrics in $\clU_1(\hilb),$ whichever is more convenient.

\begin{cor} The metric space $\brs{\clU_1(\hilb),d_1}$ is separable and complete.
\end{cor}
\begin{proof} Since the metric space $\brs{\clU_1(\hilb),\norm{U_1-U_2}_1}$ is separable and complete,
this follows from Lemma \ref{L: d1 and norm1 are equiv}
\end{proof}

In regard of the following theorem see also \cite[Proposition 3.6 and Lemma 3.2]{Pu01FA}.
\begin{thm} \label{T: spec is continuous} The mapping
$$
  \spec \colon \clU_1(\hilb) \to \euS_1(\mbT)
$$
is continuous. Moreover,
$$
  d(\spec(U_1),\spec(U_2)) \leq d_1\brs{U_1, U_2}.
$$
\end{thm}
\begin{proof}
(A) It follows from Lemma \ref{L: spec distance estimate}, that the theorem holds for finite-dimensional operators $U_1$ and $U_2.$

(B) Any operator $U$ from $\clU_1(\hilb)$ can be approximated (in $\clL_1$-norm, or, which is the same, in $d_1$-metric) by finite-dimensional operators from
$\clU_1(\hilb),$ which are constructed by throwing away all small enough summands in the Schmidt representation of $U-1,$
so that, in particular, the spectrum of the finite-dimensional approximations are subsets (in rigged sense) of the spectrum of $U.$

Let $\eps>0.$
For any $U_1$ and $U_2$ from $\clU_1(\hilb)$ there exist the above mentioned finite-dimensional operators
$U_1'$ and $U_2' \in \clU_1(\hilb)$ such that
$$
  d_1\brs{U_1,U_1'} < \frac \eps 2 \quad \text{and} \quad d_1\brs{U_2,U_2'} < \frac \eps 2,
$$
so that, by the triangle inequality,
\begin{equation} \label{F: d1(U1',U2')<...}
  d_1\brs{U_1',U_2'} < d_1\brs{U_1,U_2} + \eps.
\end{equation}
Let $N = \min\brs{\rank(U_1'-1),\rank(U_2'-1)}.$ If $N$ is large enough then
$$d(\spec(U_j),\spec(U_j')) < \eps /2, \ \ j =1,2.$$
It follows from this and the triangle inequality that
$$
  d(\spec(U_1),\spec(U_2)) < d(\spec(U_1'),\spec(U_2')) + \eps.
$$
It follows from this and (A) that
$$
  d(\spec(U_1),\spec(U_2)) < d_1\brs{U_1',U_2'} + \eps.
$$
Combining this with (\ref{F: d1(U1',U2')<...}), we obtain
$$
  d(\spec(U_1),\spec(U_2)) < d_1\brs{U_1,U_2} + 2\eps.
$$
Since $\eps>0$ is arbitrary, the claim follows.
\end{proof}

\subsection{Continuous paths of unitary operators}
As an immediate consequence of Theorem \ref{T: Selection Thm}, we have the operator version of that theorem.
\begin{thm} \label{T: Operator Selection Thm}
Let $U \colon [0,1] \to \clU_1(\hilb)$ be a continuous map, such that $U(0) = 1.$
It is possible to enumerate (counting multiplicities) eigenvalues
$e^{i\theta_j(x)}, j=1,2,\ldots,$ of $U(x)$ in such a way that the functions $\theta_j(x)$
are continuous and $\theta_j(0) = 0$ for all $j=1,2,\ldots$
\end{thm}
\begin{proof} This follows directly from Theorem \ref{T: spec is continuous} and Theorem \ref{T: Selection Thm}.
\end{proof}

\begin{defn} $\mu$-invariant of a continuous path of operators $U \colon [a,b] \to \clU_1(\hilb),$ such that $U(a) = 1,$
is a function of the angle variable $\theta \in (0,2\pi)$ defined by
$$
  \mu(\theta; U) = \mu(\theta; \spec\, U) \ \in \mbZ. 
$$
That is, $\mu$-invariant of the path of unitary operators $U(r)$ is, by definition, the $\mu$-invariant of the path $\spec\, U(r) \in \euS_1(\mbT).$
\end{defn}
If functions $\theta_j(x)$ are chosen as in Theorem \ref{T: Operator Selection Thm}, then we have
$$
  \mu(\theta; U) = - \sum_{j=1}^\infty \SqBrs{\frac{\theta - \theta_j(b)}{2\pi}} \ \in \mbZ.
$$

The above sum is finite for all $\theta \in (0,2\pi).$ It also follows that the sum is independent from
rearrangement of numbers $\theta_j(x).$

\begin{prop} Let $U$ be a unitary operator from $\clU_1(\hilb).$
If two paths $U_1, U_2 \colon [0,1] \to \clU_1(\hilb),$ which connect $1$ and $U \in \clU_1(\hilb),$
are homotopic, then their $\mu$-invariants coincide:
$$
  \mu(\cdot; U_1) = \mu(\cdot; U_2).
$$
\end{prop}
\begin{proof} This follows from Theorem \ref{T: mu-invariant is homotopically invariant}.
\end{proof}

\begin{rems*} It is known (Kuiper's theorem) that the unitary group of an infinite dimensional Hilbert space
is homotopically equivalent to a point. At the same time, the group $\clU_1(\hilb)$ is
homotopically non-trivial, as it follows from Corollary \ref{C: fundamental group of euS(T)} and Theorem \ref{T: spec is continuous}.
\end{rems*}


\section{Preliminaries on operators on a framed Hilbert space}
Details and proofs, regarding the material of this subsection, can be found in \cite{Az3v4}.

Recall that a \emph{frame} (cf. \cite{Az3v4}) in a Hilbert space $\hilb$ is a Hilbert-Schmidt operator with trivial kernel~$F,$
acting from $\hilb$ to possibly another Hilbert space $\clK,$ of the form
$$
  F = \sum_{j=1}^\infty \kappa_j \scal{\phi_j}{\cdot}\psi_j,
$$
where $(\kappa_j)$ is the sequence of $s$-numbers of~$F,$ $(\phi_j)$ is an orthonormal basis of $\hilb$
and $(\psi_j)$ is an orthonormal basis of $\clK.$

Let $\euD$ be the manifold of finite linear combinations of vectors $\phi_1,\phi_2,\ldots .$
The frame~$F,$ considered as Hilbert-Schmidt rigging (cf. \cite{BeShu}), generates a pair of Hilbert spaces
$\hilb_1$ and $\hilb_{-1}$ with scalar products $\scal{\cdot}{\cdot}_{\hilb_1}$ and
$\scal{\cdot}{\cdot}_{\hilb_{-1}}$ respectively, defined by formula
$$
  \scal{f}{g}_{\hilb_{\alpha}} = \scal{\abs{F}^{-\alpha}f}{\abs{F}^{-\alpha}g}, \ \ f,g \in \euD.
$$
The Hilbert space $\hilb_\alpha$ is the completion of $\euD$ endowed with the scalar product $\scal{\cdot}{\cdot}_{\hilb_{\alpha}}.$
There is a natural pairing $\scal{f}{g}_{1,-1}$ for $f \in \hilb_1$ and $g \in \hilb_{-1},$
and natural Hilbert-Schmidt inclusions
$$
  \hilb_1 \subset \hilb \subset \hilb_{-1}.
$$
The operator $\abs{F},$ considered as an operator $\hilb_{\alpha-1} \to \hilb_{\alpha},$ $\alpha = 0,1,$ is unitary.
So, it gives a natural isomorphism of Hilbert spaces
\begin{equation} \label{F: |F|:hilb equiv hilb1}
  \abs{F} \colon \hilb_{\alpha-1} \cong \hilb_{\alpha}.
\end{equation}
If $f,g \in \hilb_1 \subset \hilb,$ then $\scal{f}{g}_{1,-1} = \scal{f}{g}.$

Let $H_0$ be a self-adjoint operator on a framed Hilbert space $(\hilb,F).$
The standard set of full Lebesgue measure $$\LambHF{H_0},$$ associated with a frame~$F,$
is defined as the set of all those points $\lambda \in \mbR$ for which
\\ \mbox{ } \quad (1) the operator
$FR_{\lambda+iy}(H_0)F^*$ has a limit~$FR_{\lambda+i0}(H_0)F^*$ in $\clL_2$-norm as $y \to 0^+$
and
\\ \mbox{ } \quad (2) the operator
$F \Im R_{\lambda+iy}(H_0) F^*$ has a limit~$F \Im R_{\lambda+i0}(H_0)F^*$ in $\clL_1$-norm as~\mbox{$y \to 0^+.$}

\noindent That the set $\LambHF{H_0}$ is indeed a full set follows from the limiting absorption principle,
cf. \cite[Theorems 6.1.5, 6.1.9]{Ya}.

It follows from this definition, that for any $\lambda \in \LambHF{H_0}$ the limit
$$
  R_{\lambda + i0}(H_0) \colon \hilb_1 \to \hilb_{-1}
$$
exists in the Hilbert-Schmidt norm and the limit
$$
  \Im R_{\lambda + i0}(H_0) \colon \hilb_1 \to \hilb_{-1}
$$
exists in the trace-class norm.

For any $\lambda \in \Lambda(H_0;F)$ one can define the non-negative trace-class matrix
$$
  \phi(\lambda) := (\phi_{ij}(\lambda)) = \frac 1\pi \brs{\kappa_i\kappa_j \scal{\phi_i}{\Im R_{\lambda+i0}(H_0) \phi_j}}.
$$
The value $\phi_j(\lambda)$ of the vector $\phi_j$ at $\lambda \in \LambHF{H_0}$
is defined by the formula
$$
  \phi_j(\lambda) = \kappa_j^{-1} \eta_j(\lambda).
$$
The (fiber) Hilbert space $\hlambda = \hlambda(H_0; F) \subset \ell_2$ is, by definition, the closure of the linear span
of $\phi_j(\lambda),$ $j=1,2,\ldots$

The family $\set{\phi_j(\lambda), \ j=1,2,\ldots}$ of vector-functions form a measurability base of the direct integral of Hilbert spaces
$$
  \euH := \int_{\LambHF{H_0}}^\oplus \mathfrak h_\lambda\,d\lambda,
$$
where the case of~$\dim \hlambda = 0$ is not excluded.

For any $f \in \hilb_1(F),$ (so that
$$
  f = \sum\limits_{j=1}^\infty \kappa_j \beta_j \phi_j,
$$
where $(\beta_j) \in \ell_2,$) 
we define the value $f(\lambda)$ of $f$ at $\lambda$ by the formula
$$
  f(\lambda) := \euE_\lambda f := \sum_{j=1}^\infty \kappa_j \beta_j \phi_j(\lambda).
$$
The series here absolutely converges in $\ell_2.$
The operator
$$
  \euE \colon \hilb_1 \to \euH,
$$
defined by the formula $(\euE f)(\lambda) = \euE_\lambda f,$ is Hilbert-Schmidt,
and, considered as an operator from~$\hilb$ to $\euH,$ it is bounded,
vanishes on the singular subspace~$\hilb^{(s)}(H_0)$ of~$H_0,$ it is isometric on the absolutely continuous subspace~$\hilb^{(a)}(H_0)$
of~$H_0$ with the range $\euH$
and is diagonalizing for~$H_0.$ That is, given a frame $F$ in the Hilbert space $\hilb,$
the operator $\euE$ gives a natural isomorphism of the absolutely continuous subspace of $\hilb$
to the direct integral $\euH:$
$$
  \euE \colon \hilb^{(a)} \cong \euH,
$$
and
$$
  \euE_\lambda (H_0f) = \lambda\euE_\lambda (f)
$$
for a.e. $\lambda \in \Lambda(H_0;F).$

Let $H_0$ be a self-adjoint operator on a framed Hilbert space $(\hilb,F)$
and let $V$ be a trace class self-adjoint operator on $\hilb,$
of the form
$$
  V = F^*JF,
$$
where $J \in \clB(\clK).$ Using natural isomorphisms (\ref{F: |F|:hilb equiv hilb1}),
such an operator $V$ can also be considered
as a bounded operator
$$
  V \colon \hilb_{-1} \to \hilb_1.
$$
Let $$H_r = H_0+V_r, \ r \in \mbR,$$ where
$$
  V_r = F^*J_rF
$$
and $J_r$ is a piecewise analytic path of bounded operators in $\clK.$

Let $\lambda \in \LambHF{H_0}.$
The \emph{resonance set} $R(\lambda,H_0,V;F)$
can be defined as the set of all those $r \in \mbR$ for which $\lambda \in \LambHF{H_r}.$
Another way to characterize the resonance set is this
$$
  R(\lambda,H_0,V;F) = \set{r \in \mbR \colon \text{the operator} \ 1 + J_rFR_{\lambda+i0}(H_0)F^* \ \text{is not invertible}}.
$$
This second description of the resonance set $R(\lambda,H_0,V;F)$ implies that it is  a discrete subset of $\mbR,$
cf e.g. \cite{Az3v4}.

For any $\lambda \in \Lambda(H_0;F) \cap \Lambda(H_r;F),$ one can define operators
\begin{equation} \label{F: def of w pm}
  w_\pm(\lambda) \colon \hlambda(H_0) \to \hlambda(H_r)
\end{equation}
as an (unique) operator, which satisfies the formula
$$
  \scal{\euE_\lambda(H_1) f}{w_\pm(\lambda; H_1,H_0) \euE_\lambda(H_0) g} = \scal{f}{\mathfrak a_\pm(\lambda;H_1,H_0)g}_{1,-1},
$$
where the operator
$$
  \mathfrak a_\pm(\lambda;H_1,H_0) \colon \hilb_1 \to \hilb_{-1}
$$
is defined by the formula
$$
  \mathfrak a_\pm(\lambda;H_1,H_0) := \SqBrs{1- R_{\lambda \mp i0}(H_1)V} \cdot \frac 1\pi \Im R_{\lambda + i0}(H_0).
$$
The operators (\ref{F: def of w pm}) are unitary operators.

So, for any $\lambda \in \Lambda(H_0;F) \cap \Lambda(H_r;F)$ one can define the scattering matrix by the formula
$$
  S(\lambda;H_r,H_0) = w^*_+(\lambda;H_r,H_0)w_-(\lambda;H_r,H_0).
$$
For any $\lambda \in \LambHF{H_0} \cap \LambHF{H_r}$ the stationary formula
\begin{equation} \label{F: stationary formula for S}
  S(\lambda;H_r,H_0) = 1_\lambda - 2\pi i \euE_\lambda V_r(1 + R_{\lambda+i0}(H_0)V_r)^{-1} \euE^\diamondsuit_\lambda
\end{equation}
holds, where $\euE^\diamondsuit_\lambda = \abs{F}^{-2} \euE^*_\lambda  \colon \hlambda \to \hilb_{-1}(F).$

\section{Singular spectral shift and Pushnitski $\mu$-invariant}
This section very heavily relies on my paper \cite{Az3v4}.
I refer to this paper, instead of giving all the necessary definitions and results from it here.

\subsection{Absolutely continuous part of Pushnitski $\mu$-invariant}

The aim of this subsection is to introduce the absolutely continuous part $\mua(\theta,\lambda;H_r,H_0)$
of Pushnitski $\mu$-invariant (cf. \cite{Pu01FA}) and to prove Theorem \ref{T: xia = 1/2pi int}.

Let $H_0$ be a self-adjoint operator on a Hilbert space with frame operator $F.$
Let $\set{V_r}$ be a continuous piecewise analytic path of trace-class self-adjoint operators,
such that \cite[Assumption 5.1]{Az3v4} holds and $V_0 = 0.$ Let $H_r = H_0+V_r.$ Note that for a path
$V_r = rV$ with any trace-class self-adjoint operator $V$ there exists a frame $F$ such that the assumption holds.

Let $\lambda \in \LambHF{H_0}.$ It is proved in \cite{Az3v4},
that in this case the scattering matrix $S(\lambda; H_r,H_0)$
exists for all $r \in \mbR$ except a discrete resonance set $R(\lambda; H_0,V;F),$
and that the function
$$
  \mbR \ni r \mapsto S(\lambda; H_r,H_0)
$$
admits analytic continuation to a neighbourhood of $\mbR.$
Also, it follows from (\ref{F: stationary formula for S}) that $S(\lambda; H_r,H_0)$ takes values in $1 + \clL_1(\hlambda).$
This allows to introduce the $\mu$-invariant of the scattering matrix $S(\lambda; H_r,H_0).$

\begin{defn} \label{D: a.c. Push. inv-t}
Let $\lambda \in \LambHF{H_0}.$ The \emph{absolutely continuous part of Pushnitski $\mu$-invariant}
of the pair $(H_0,H_r)$ is the $\mu$-invariant of the continuous path
$$
  [0,s] \ni r \mapsto S(\lambda; H_r, H_0) \in 1 + \clL_1(\hlambda(H_0)).
$$
The absolutely continuous part of Pushnitski $\mu$-invariant will be denoted by $\mua(\theta,\lambda; H_r,H_0).$
\end{defn}

By Theorem \ref{T: Operator Selection Thm},
it is possible to choose eigenvalues $e^{i\theta^*_j(\lambda,r)} \in \mbT,$ $j=1,2,\ldots$ of the scattering matrix
$S(\lambda;H_r,H_0)$ in such a way that for all $j=1,2,\ldots$ \ $\theta_j^*(0) = 0$ and $\theta^*_j(\lambda,\cdot)$ is continuous.
Numbers $\theta^*_j(\lambda,r)$ are also called scattering phases.

We shall always assume such a choice of $\theta_j^*$'s.

By definition of the $\mu$-invariant, for all $\theta \in (0,2\pi)$ and all $\lambda \in \LambHF{H_0},$
\begin{equation} \label{F: def of a.c. mu}
  \mua(\theta,\lambda; H_r, H_0) = - \sum_{j=1}^\infty \SqBrs{\frac{\theta - \theta^*_j(\lambda,r)}{2\pi}} \ \ \in \mbZ,
\end{equation}
where $[x]$ denotes the integer part of $x \in \mbR.$ The right hand side does not depend on the choice of functions $\set{\theta_j^*(\lambda,r)}.$

Note that, by Lemma \ref{L: the sum is continuous}, for any $\lambda \in \LambHF{H_0}$
\begin{equation} \label{F: sum theta j(r) is continuous}
   \text{the function} \ \
      \mbR \ni r \mapsto \sum\limits_{j=1}^\infty \theta^*_j(\lambda,r) \ \ \text{is continuous.}
\end{equation}

The following lemma follows directly from Lemma \ref{L: SS3 int mu = sum theta j}.
\begin{lemma} \label{L: int mua = sum theta j}
Let $\lambda \in \LambHF{H_0}.$
The function $\theta \mapsto \mua(\theta,\lambda;H_r,H_0)$ is summable and the equality
  \begin{equation} \label{F: int of ac mu = sum of theta*}
    \int_0^{2\pi} \mua(\theta,\lambda; H_r,H_0)\,d\theta = \sum_{j=1}^\infty \theta^*_j(\lambda,r)
  \end{equation}
  holds.
\end{lemma}
\begin{thm} \label{T: xia = 1/2pi int} Let $\lambda \in \LambHF{H_0} \cap \LambHF{H_1}.$
The equality
\begin{equation} \label{F: xia = 1/2pi int}
  \xia(\lambda; H_1,H_0) = - \frac 1{2\pi} \int_0^{2\pi} \mua(\theta,\lambda; H_1,H_0)\,d\theta
\end{equation}
holds.
\end{thm}
\begin{proof} By the Lidskii theorem (see e.g. \cite{GK,SimTrId})
\begin{equation*}
  \det S(\lambda;H_r,H_0) = \prod\limits_{j=1}^\infty e^{i\theta^*_j(\lambda,r)} = \exp\brs{i\sum\limits_{j=1}^\infty \theta_j^*(r,\lambda)}.
\end{equation*}
It follows from this, (\ref{F: sum theta j(r) is continuous}) and \cite[Theorem 9.8]{Az3v4} that
$$
  \xia(\lambda; H_r,H_0) = - \frac 1{2\pi} \sum\limits_{j=1}^\infty \theta_j^*(r,\lambda).
$$
Now, Lemma \ref{L: int mua = sum theta j} completes the proof.
\end{proof}

{\bf Remark}. Note that definition of the $\mu$-invariant does not depend on the choice of the path $\set{H_r}$
connecting the end-point operators. Indeed, any two $\mu$-invariants must differ by a constant integer.
Since definition of $\xia$ is path-independent \cite{Az3v4}, it follows from the previous theorem
that different paths give the same $\mu$-invariant.

\subsection{Pushnitski $\mu$-invariant}
Let $H_0$ be a self-adjoint operator on $\hilb,$
let $F$ be a frame operator on $\hilb,$ let $V_r \in F^*J_rF,$ where $J_r \in \clB(\clK),$
and let $\set{H_r=H_0+V_r}$ satisfy Assumption 5.1 of \cite{Az3v4}. In this subsection $r$ will be fixed,
so it is in fact not important what the path $\set{H_r}$ is which connects $H_0$ and $H_1.$

Let $z \in \mbC,$ $\Im z > 0.$ Following \cite{Pu01FA}, we define the
$\tilde S$-function by the formula
\begin{equation} \label{F: def of tilde S(z,r)}
 \begin{split}
   \tilde S(z,r) & = \tilde S(z; H_r,H_0;F)
     \\ & = 1 - 2i \sqrt{\Im T_0(z)} J_r (1 + T_0(z)J_r)^{-1} \sqrt{\Im T_0(z)} \ \ \in \ 1 + \clL_1(\clK),
 \end{split}
\end{equation}
where
$$
  T_0(z) = F R_z(H_0) F^*.
$$
It is not difficult to verify that $\tilde S(z;
H_r,H_0;F)$ is a unitary operator. Hence,
$$
  \tilde S(z; H_r,H_0;F) \in \clU_1(\clK).
$$
If $\lambda \in \Lambda(H_0,F) \cap \Lambda(H_r,F),$ then the limit
\begin{equation} \label{F: tilde S(l+i0,r)}
 \begin{split}
   \tilde S(\lambda+i0,r) & = \tilde S(\lambda+i0; H_r,H_0;F)
     \\ & = 1 - 2i \sqrt{\Im T_0(\lambda+i0)} J_r (1 + T_0(\lambda+i0)J_r)^{-1} \sqrt{\Im T_0(\lambda+i0)} \ \ \in \ \clU_1(\clK)
 \end{split}
\end{equation}
exists in $\clL_1(\clK)$-norm: from $\lambda \in \Lambda(H_0,F)$ it follows existence of $\Im T_0(\lambda+i0)$
in $\clL_2(\clK)$-norm and from $\lambda \in \Lambda(H_r,F)$ it follows that the operator $(1 + T_0(\lambda+i0)J_r)^{-1}$
is invertible.

When $y \to +\infty,$ the operator $\tilde S(\lambda+iy,r)$ goes to $1.$
So, we have a continuous (in fact, real-analytic) path of unitary operators in $\clU_1(\clK):$
\begin{equation} \label{F: tilde S on [0,infty]}
  \tilde S(\lambda+i\,\cdot,r) \colon [0,\infty] \to \clU_1(\clK).
\end{equation}
\begin{defn} \label{D: Push. inv-t}
Let $\lambda \in \LambHF{H_0}\cap \LambHF{H_r}.$ \emph{Pushnitski $\mu$-invariant}
of the pair $(H_0,H_r)$ is the $\mu$-invariant of the continuous path (\ref{F: tilde S on [0,infty]}).
Pushnitski $\mu$-invariant will be denoted by $\mu(\theta,\lambda; H_r,H_0).$
\end{defn}
The following theorem was proved in \cite{Pu01FA}.
\begin{thm} (Pushnitski formula) \label{T: xi = -average of mu} For a.e. $\lambda \in \LambHF{H_0}\cap \LambHF{H_r}$ the equality
$$
  \xi(\lambda; H_r,H_0) = - \frac 1{2\pi} \int_0^{2\pi} \mu(\theta,\lambda; H_r,H_0)
$$
holds.
\end{thm}
It is possible to give another proof of this theorem, which follows a proof in \cite{Az2}.
This proof will appear elsewhere.

\subsection{The singular part of Pushnitski $\mu$-invariant}
The scattering matrix $S(\lambda; H_r,H_0)$ for $\lambda \in \LambHF{H_0} \cap \LambHF{H_r}$ is a unitary operator of the class
$1 + \clL_1(\hlambda).$ So, the spectrum of $S(\lambda; H_r,H_0)$ is discrete with only possible accumulation point at $1,$
and its eigenvalues belong to the unit circle $\mbT.$ The eigenvalues of $S(\lambda; H_r,H_0)$ (so called scattering phases)
can be send to $1$ in two essentially different ways. The first way is to connect $S(\lambda; H_r,H_0)$  with the identity operator
by letting the coupling constant $r$ move from $1$ to $0.$ This is possible to do, since $S(\lambda; H_r,H_0)$
is continuous for all $r \in \mbR.$ Now, observe that
the operators $S(\lambda; H_r,H_0)$ and $\tilde S(\lambda+i0; H_r,H_0)$ have the same eigenvalues.
So, the second way to send scattering phases to $1$ is to move $y$ from $0$ to $+\infty$ in $\tilde S.$ In both ways,
scattering phases go to $1$ continuously. Nevertheless, it is possible that these two ways
are not homotopic in that an eigenvalue can make a different number of windings around the unit circle
as it is sent to $1.$ The Pushnitski invariant $\mu(\theta,\lambda;H_r,H_0)$ and its absolutely continuous part
$\mua(\theta,\lambda;H_r,H_0)$ measure the spectral flow of the scattering phases through $e^{i\theta}$
in two different ways, corresponding to the above mentioned two ways of connecting the scattering phases
with $1.$ So, the difference $\mu(\theta,\lambda;H_r,H_0) - \mua(\theta,\lambda;H_r,H_0)$ presumably should not depend on $\theta.$
This difference measures the difference of winding numbers.

\begin{defn} Let $\lambda \in \LambHF{H_0} \cap \LambHF{H_r}.$
The singular part of Pushnitski $\mu$-invariant of the pair $(H_0,H_r)$ is the function
$$
  \mus(\theta,\lambda;H_r,H_0) := \mu(\theta,\lambda;H_r,H_0) - \mua(\theta,\lambda;H_r,H_0).
$$
\end{defn}

\begin{thm} \label{T: mus(theta)=const}
The singular part of Pushnitski $\mu$-invariant $\mus(\theta,\lambda; H_r,H_0)$ does not depend on the angle variable $\theta.$
\end{thm}
\begin{proof}
(A) We use the well-known equality $\spec(AB) \cup \set{0} = \spec(BA)\cup \set{0}.$
This equality, the stationary formula (\ref{F: stationary formula for S}) for the scattering matrix,
formula (\ref{F: tilde S(l+i0,r)})
and the equality
$$
  \euE_\lambda^\diamondsuit(H_0)\euE_\lambda(H_0) = \frac 1 \pi \Im R_{\lambda+i0}(H_0)
$$
imply that the spectra of operators $S(\lambda; H_r,H_0)$ and $\tilde S (\lambda; H_r,H_0;F)$ coincide
as elements of $\euS_1(\mbT).$

(B) It follows from (A) and Corollary \ref{C: mu-inv. = const} that the singular part of the $\mu$-invariant does not depend on $\theta.$

\end{proof}

The proof, given here, unlike the one given in \cite{Az2},
does not use a study of the behaviour of specific eigenvalue-functions of the $\tilde S$-operator.
Still, investigation of the eigenvalue-functions is of separate interest. It will be given elsewhere.

\begin{cor} \label{C: xis = -mus} The following formula holds
$$
  \xis(\lambda) = - \mus(\lambda).
$$
Consequently, the singular part of the spectral shift function $\xis(\lambda)$ is integer-valued.
\end{cor}
\begin{proof} 
It follows from Theorems \ref{T: xia = 1/2pi int} and \ref{T: xi = -average of mu} that
$$
  \xis(\lambda) = -\frac 1{2\pi}\int_0^{2\pi} \mus(\theta,\lambda;H_r,H_0)\,d\theta.
$$
So, Theorem \ref{T: mus(theta)=const} completes the proof.
\end{proof}

\begin{cor} {\rm (Birman-\Krein\ formula)} \ Let $H_0$ be a self-adjoint operator
on a framed Hilbert space $(\hilb,F)$
and let $V$ be a trace class self-adjoint operator, such that $V$ is bounded from $\hilb_{-1}$ to $\hilb_1.$
Then for all $\lambda \in \LambHF{H_0} \cap \LambHF{H_1}$
$$
  e^{-2\pi i \xi(\lambda;H_1,H_0)} = \det S(\lambda;H_1,H_0),
$$
where $H_1 = H_0 + V.$
\end{cor}
\begin{proof} This follows from \cite[Corollary 9.9]{Az3v4} and Corollary \ref{C: xis = -mus}.
\end{proof}

A non-trivial example of a pair $H_0$ and $H_1,$ for which $\xis \neq 0$ on the absolutely continuous spectrum of $H_0,$
can be found in \cite{Az6}.

%

%

\margcom{Edit \\ references}

\input MyListOfRef

\end{document}

%% file: Macros.tex
    \newtheorem{thm}{Theorem}                     [section]
    \newtheorem{thm*}{Theorem}
    \newtheorem{prop}[thm]{Proposition}
    \newtheorem{lemma}[thm]{Lemma}
    \newtheorem{cor}[thm]{Corollary}

    \newtheorem{lemma*}{Lemma}    

    \newtheorem{defn}[thm]{Definition}                 
    \newtheorem{example}[thm]{Example}                 
    \newtheorem{rems}[thm]{Remark}                     
    \newtheorem{rems*}{Remark}   

\newcommand{\ndef}{\newcommand*}
\def\rndef{\renewcommand}

\ndef{\myaddress}[1]{\begin{center} \it\small #1 \end{center}}

\ndef{\clA}{{\mathcal A}} \ndef{\rmA}{{\mathrm A}} \ndef{\mbA}{{\mathbb A}} \ndef{\bfA}{{\mathbf A}} \ndef{\euA}{{\EuScript A}} \ndef{\frA}{{\mathfrak A}}
\ndef{\clB}{{\mathcal B}} \ndef{\rmB}{{\mathrm B}} \ndef{\mbB}{{\mathbb B}} \ndef{\bfB}{{\mathbf B}} \ndef{\euB}{{\EuScript B}} \ndef{\frB}{{\mathfrak B}}
\ndef{\clC}{{\mathcal C}} \ndef{\rmC}{{\mathrm C}} \ndef{\mbC}{{\mathbb C}} \ndef{\bfC}{{\mathbf C}} \ndef{\euC}{{\EuScript C}} \ndef{\frC}{{\mathfrak C}}
\ndef{\clD}{{\mathcal D}} \ndef{\rmD}{{\mathrm D}} \ndef{\mbD}{{\mathbb D}} \ndef{\bfD}{{\mathbf D}} \ndef{\euD}{{\EuScript D}} \ndef{\frD}{{\mathfrak D}}
\ndef{\clE}{{\mathcal E}} \ndef{\rmE}{{\mathrm E}} \ndef{\mbE}{{\mathbb E}} \ndef{\bfE}{{\mathbf E}} \ndef{\euE}{{\EuScript E}} \ndef{\frE}{{\mathfrak E}}
\ndef{\clF}{{\mathcal F}} \ndef{\rmF}{{\mathrm F}} \ndef{\mbF}{{\mathbb F}} \ndef{\bfF}{{\mathbf F}} \ndef{\euF}{{\EuScript F}} \ndef{\frF}{{\mathfrak F}}
\ndef{\clG}{{\mathcal G}} \ndef{\rmG}{{\mathrm G}} \ndef{\mbG}{{\mathbb G}} \ndef{\bfG}{{\mathbf G}} \ndef{\euG}{{\EuScript G}} \ndef{\frG}{{\mathfrak G}}
\ndef{\clH}{{\mathcal H}} \ndef{\rmH}{{\mathrm H}} \ndef{\mbH}{{\mathbb H}} \ndef{\bfH}{{\mathbf H}} \ndef{\euH}{{\EuScript H}} \ndef{\frH}{{\mathfrak H}}
\ndef{\clI}{{\mathcal I}} \ndef{\rmI}{{\mathrm I}} \ndef{\mbI}{{\mathbb I}} \ndef{\bfI}{{\mathbf I}} \ndef{\euI}{{\EuScript I}} \ndef{\frI}{{\mathfrak I}}
\ndef{\clJ}{{\mathcal J}} \ndef{\rmJ}{{\mathrm J}} \ndef{\mbJ}{{\mathbb J}} \ndef{\bfJ}{{\mathbf J}} \ndef{\euJ}{{\EuScript J}} \ndef{\frJ}{{\mathfrak J}}
\ndef{\clK}{{\mathcal K}} \ndef{\rmK}{{\mathrm K}} \ndef{\mbK}{{\mathbb K}} \ndef{\bfK}{{\mathbf K}} \ndef{\euK}{{\EuScript K}} \ndef{\frK}{{\mathfrak K}}
\ndef{\clL}{{\mathcal L}} \ndef{\rmL}{{\mathrm L}} \ndef{\mbL}{{\mathbb L}} \ndef{\bfL}{{\mathbf L}} \ndef{\euL}{{\EuScript L}} \ndef{\frL}{{\mathfrak L}}
\ndef{\clM}{{\mathcal M}} \ndef{\rmM}{{\mathrm M}} \ndef{\mbM}{{\mathbb M}} \ndef{\bfM}{{\mathbf M}} \ndef{\euM}{{\EuScript M}} \ndef{\frM}{{\mathfrak M}}
\ndef{\clN}{{\mathcal N}} \ndef{\rmN}{{\mathrm N}} \ndef{\mbN}{{\mathbb N}} \ndef{\bfN}{{\mathbf N}} \ndef{\euN}{{\EuScript N}} \ndef{\frN}{{\mathfrak N}}
\ndef{\clO}{{\mathcal O}} \ndef{\rmO}{{\mathrm O}} \ndef{\mbO}{{\mathbb O}} \ndef{\bfO}{{\mathbf O}} \ndef{\euO}{{\EuScript O}} \ndef{\frO}{{\mathfrak O}}
\ndef{\clP}{{\mathcal P}} \ndef{\rmP}{{\mathrm P}} \ndef{\mbP}{{\mathbb P}} \ndef{\bfP}{{\mathbf P}} \ndef{\euP}{{\EuScript P}} \ndef{\frP}{{\mathfrak P}}
\ndef{\clQ}{{\mathcal Q}} \ndef{\rmQ}{{\mathrm Q}} \ndef{\mbQ}{{\mathbb Q}} \ndef{\bfQ}{{\mathbf Q}} \ndef{\euQ}{{\EuScript Q}} \ndef{\frQ}{{\mathfrak Q}}
\ndef{\clR}{{\mathcal R}} \ndef{\rmR}{{\mathrm R}} \ndef{\mbR}{{\mathbb R}} \ndef{\bfR}{{\mathbf R}} \ndef{\euR}{{\EuScript R}} \ndef{\frR}{{\mathfrak R}}
\ndef{\clS}{{\mathcal S}} \ndef{\rmS}{{\mathrm S}} \ndef{\mbS}{{\mathbb S}} \ndef{\bfS}{{\mathbf S}} \ndef{\euS}{{\EuScript S}} \ndef{\frS}{{\mathfrak S}}
\ndef{\clT}{{\mathcal T}} \ndef{\rmT}{{\mathrm T}} \ndef{\mbT}{{\mathbb T}} \ndef{\bfT}{{\mathbf T}} \ndef{\euT}{{\EuScript T}} \ndef{\frT}{{\mathfrak T}}
\ndef{\clU}{{\mathcal U}} \ndef{\rmU}{{\mathrm U}} \ndef{\mbU}{{\mathbb U}} \ndef{\bfU}{{\mathbf U}} \ndef{\euU}{{\EuScript U}} \ndef{\frU}{{\mathfrak U}}
\ndef{\clV}{{\mathcal V}} \ndef{\rmV}{{\mathrm V}} \ndef{\mbV}{{\mathbb V}} \ndef{\bfV}{{\mathbf V}} \ndef{\euV}{{\EuScript V}} \ndef{\frV}{{\mathfrak V}}
\ndef{\clW}{{\mathcal W}} \ndef{\rmW}{{\mathrm W}} \ndef{\mbW}{{\mathbb W}} \ndef{\bfW}{{\mathbf W}} \ndef{\euW}{{\EuScript W}} \ndef{\frW}{{\mathfrak W}}
\ndef{\clX}{{\mathcal X}} \ndef{\rmX}{{\mathrm X}} \ndef{\mbX}{{\mathbb X}} \ndef{\bfX}{{\mathbf X}} \ndef{\euX}{{\EuScript X}} \ndef{\frX}{{\mathfrak X}}
\ndef{\clY}{{\mathcal Y}} \ndef{\rmY}{{\mathrm Y}} \ndef{\mbY}{{\mathbb Y}} \ndef{\bfY}{{\mathbf Y}} \ndef{\euY}{{\EuScript Y}} \ndef{\frY}{{\mathfrak Y}}
\ndef{\clZ}{{\mathcal Z}} \ndef{\rmZ}{{\mathrm Z}} \ndef{\mbZ}{{\mathbb Z}} \ndef{\bfZ}{{\mathbf Z}} \ndef{\euZ}{{\EuScript Z}} \ndef{\frZ}{{\mathfrak Z}}

\ndef{\tA}{{\widetilde A}} \ndef{\tcA}{{\widetilde\clA}} \ndef{\ttcA}{\widetilde{\tcA}} \ndef{\sfA}{{\textsf A}} \ndef{\ttA}{\widetilde{\tA}} \ndef{\dzA}{{A^\sharp}}
\ndef{\tB}{{\widetilde B}} \ndef{\tcB}{{\widetilde\clB}} \ndef{\ttcB}{\widetilde{\tcB}} \ndef{\sfB}{{\textsf B}} \ndef{\ttB}{\widetilde{\tB}} \ndef{\dzB}{{B^\sharp}}
\ndef{\tC}{{\widetilde C}} \ndef{\tcC}{{\widetilde\clC}} \ndef{\ttcC}{\widetilde{\tcC}} \ndef{\sfC}{{\textsf C}} \ndef{\ttC}{\widetilde{\tC}} \ndef{\dzC}{{C^\sharp}}
\ndef{\tD}{{\widetilde D}} \ndef{\tcD}{{\widetilde\clD}} \ndef{\ttcD}{\widetilde{\tcD}} \ndef{\sfD}{{\textsf D}} \ndef{\ttD}{\widetilde{\tD}} \ndef{\dzD}{{D^\sharp}}
\ndef{\tE}{{\widetilde E}} \ndef{\tcE}{{\widetilde\clE}} \ndef{\ttcE}{\widetilde{\tcE}} \ndef{\sfE}{{\textsf E}} \ndef{\ttE}{\widetilde{\tE}} \ndef{\dzE}{{E^\sharp}}
\ndef{\tF}{{\widetilde F}} \ndef{\tcF}{{\widetilde\clF}} \ndef{\ttcF}{\widetilde{\tcF}} \ndef{\sfF}{{\textsf F}} \ndef{\ttF}{\widetilde{\tF}} \ndef{\dzF}{{F^\sharp}}
\ndef{\tG}{{\widetilde G}} \ndef{\tcG}{{\widetilde\clG}} \ndef{\ttcG}{\widetilde{\tcG}} \ndef{\sfG}{{\textsf G}} \ndef{\ttG}{\widetilde{\tG}} \ndef{\dzG}{{G^\sharp}}
\ndef{\tH}{{\widetilde H}} \ndef{\tcH}{{\widetilde\clH}} \ndef{\ttcH}{\widetilde{\tcH}} \ndef{\sfH}{{\textsf H}} \ndef{\ttH}{\widetilde{\tH}} \ndef{\dzH}{{H^\sharp}}
\ndef{\tI}{{\widetilde I}} \ndef{\tcI}{{\widetilde\clI}} \ndef{\ttcI}{\widetilde{\tcI}} \ndef{\sfI}{{\textsf I}} \ndef{\ttI}{\widetilde{\tI}} \ndef{\dzI}{{I^\sharp}}
\ndef{\tJ}{{\widetilde J}} \ndef{\tcJ}{{\widetilde\clJ}} \ndef{\ttcJ}{\widetilde{\tcJ}} \ndef{\sfJ}{{\textsf J}} \ndef{\ttJ}{\widetilde{\tJ}} \ndef{\dzJ}{{J^\sharp}}
\ndef{\tK}{{\widetilde K}} \ndef{\tcK}{{\widetilde\clK}} \ndef{\ttcK}{\widetilde{\tcK}} \ndef{\sfK}{{\textsf K}} \ndef{\ttK}{\widetilde{\tK}} \ndef{\dzK}{{K^\sharp}}
\ndef{\tL}{{\widetilde L}} \ndef{\tcL}{{\widetilde\clL}} \ndef{\ttcL}{\widetilde{\tcL}} \ndef{\sfL}{{\textsf L}} \ndef{\ttL}{\widetilde{\tL}} \ndef{\dzL}{{L^\sharp}}
\ndef{\tM}{{\widetilde M}} \ndef{\tcM}{{\widetilde\clM}} \ndef{\ttcM}{\widetilde{\tcM}} \ndef{\sfM}{{\textsf M}} \ndef{\ttM}{\widetilde{\tM}} \ndef{\dzM}{{M^\sharp}}
\ndef{\tN}{{\widetilde N}} \ndef{\tcN}{{\widetilde\clN}} \ndef{\ttcN}{\widetilde{\tcN}} \ndef{\sfN}{{\textsf N}} \ndef{\ttN}{\widetilde{\tN}} \ndef{\dzN}{{N^\sharp}}
\ndef{\tO}{{\widetilde O}} \ndef{\tcO}{{\widetilde\clO}} \ndef{\ttcO}{\widetilde{\tcO}} \ndef{\sfO}{{\textsf O}} \ndef{\ttO}{\widetilde{\tO}} \ndef{\dzO}{{O^\sharp}}
\ndef{\tP}{{\widetilde P}} \ndef{\tcP}{{\widetilde\clP}} \ndef{\ttcP}{\widetilde{\tcP}} \ndef{\sfP}{{\textsf P}} \ndef{\ttP}{\widetilde{\tP}} \ndef{\dzP}{{P^\sharp}}
\ndef{\tQ}{{\widetilde Q}} \ndef{\tcQ}{{\widetilde\clQ}} \ndef{\ttcQ}{\widetilde{\tcQ}} \ndef{\sfQ}{{\textsf Q}} \ndef{\ttQ}{\widetilde{\tQ}} \ndef{\dzQ}{{Q^\sharp}}
\ndef{\tR}{{\widetilde R}} \ndef{\tcR}{{\widetilde\clR}} \ndef{\ttcR}{\widetilde{\tcR}} \ndef{\sfR}{{\textsf R}} \ndef{\ttR}{\widetilde{\tR}} \ndef{\dzR}{{R^\sharp}}
\ndef{\tS}{{\widetilde S}} \ndef{\tcS}{{\widetilde\clS}} \ndef{\ttcS}{\widetilde{\tcS}} \ndef{\sfS}{{\textsf S}} \ndef{\ttS}{\widetilde{\tS}} \ndef{\dzS}{{S^\sharp}}
\ndef{\tT}{{\widetilde T}} \ndef{\tcT}{{\widetilde\clT}} \ndef{\ttcT}{\widetilde{\tcT}} \ndef{\sfT}{{\textsf T}} \ndef{\ttT}{\widetilde{\tT}} \ndef{\dzT}{{T^\sharp}}
\ndef{\tU}{{\widetilde U}} \ndef{\tcU}{{\widetilde\clU}} \ndef{\ttcU}{\widetilde{\tcU}} \ndef{\sfU}{{\textsf U}} \ndef{\ttU}{\widetilde{\tU}} \ndef{\dzU}{{U^\sharp}}
\ndef{\tV}{{\widetilde V}} \ndef{\tcV}{{\widetilde\clV}} \ndef{\ttcV}{\widetilde{\tcV}} \ndef{\sfV}{{\textsf V}} \ndef{\ttV}{\widetilde{\tV}} \ndef{\dzV}{{V^\sharp}}
\ndef{\tW}{{\widetilde W}} \ndef{\tcW}{{\widetilde\clW}} \ndef{\ttcW}{\widetilde{\tcW}} \ndef{\sfW}{{\textsf W}} \ndef{\ttW}{\widetilde{\tW}} \ndef{\dzW}{{W^\sharp}}
\ndef{\tX}{{\widetilde X}} \ndef{\tcX}{{\widetilde\clX}} \ndef{\ttcX}{\widetilde{\tcX}} \ndef{\sfX}{{\textsf X}} \ndef{\ttX}{\widetilde{\tX}} \ndef{\dzX}{{X^\sharp}}
\ndef{\tY}{{\widetilde Y}} \ndef{\tcY}{{\widetilde\clY}} \ndef{\ttcY}{\widetilde{\tcY}} \ndef{\sfY}{{\textsf Y}} \ndef{\ttY}{\widetilde{\tY}} \ndef{\dzY}{{Y^\sharp}}
\ndef{\tZ}{{\widetilde Z}} \ndef{\tcZ}{{\widetilde\clZ}} \ndef{\ttcZ}{\widetilde{\tcZ}} \ndef{\sfZ}{{\textsf Z}} \ndef{\ttZ}{\widetilde{\tZ}} \ndef{\dzZ}{{Z^\sharp}}

\ndef{\bfa}{{\mathbf a}}
\ndef{\bfb}{{\mathbf b}}
\ndef{\bfc}{{\mathbf c}}
\ndef{\bfd}{{\mathbf d}}

\ndef{\euu}{{\EuScript u}}

  \ndef{\eps}{\varepsilon}

\let\geq\geqslant
\let\leq\leqslant

\ndef{\lims}[1]{\lim\limits_{#1}}
\ndef{\sums}[1]{\sum\limits_{#1}}
\ndef{\ints}[1]{\int_{#1}}
\ndef{\sups}[1]{\sup\limits_{#1}}
\ndef{\liminfty}[1]{\lims{#1\to\infty}}
\ndef{\suminf}[1]{\sums{#1=1}^\infty}

\ndef{\limo}[1]{\omega\mbox{-}\!\!\!\lims{#1\to\infty}}          
\ndef{\limL}[1]{\rmL\mbox{-}\!\!\!\lims{#1\to\infty}}            
\ndef{\limLOne}[1]{\clL_1\mbox{-}\!\lims{#1}}
\ndef{\tildelimo}[1]{\tilde\omega\mbox{-}\!\!\!\lims{#1\to\infty}}
\ndef{\slim}{\mathrm{s}\mbox{-}\!\!\lim}          
\ndef{\wlim}{\mathrm{w}\mbox{-}\!\lim}          

\ndef{\Aut}{\operatorname{Aut}}      
\ndef{\Ch}{\operatorname{ch}}        
\ndef{\End}{\operatorname{End}}      
\ndef{\Hom}{\operatorname{Hom}}      
\rndef{\ker}{\operatorname{ker}}      
\ndef{\coker}{\operatorname{coker}}      
\ndef{\im}{\operatorname{im}}        
\ndef{\Log}{\operatorname{Log}}      
\ndef{\OP}{\operatorname{OP}}        
\ndef{\Op}{\operatorname{Op}}        
\ndef{\Symb}{\operatorname{Symb}}    
\ndef{\Tr}{\operatorname{Tr}}        
\ndef{\Wres}{\operatorname{Wres}}    
\ndef{\cl}{\operatorname{cl}}        
\ndef{\com}{\operatorname{com}}
\ndef{\const}{\operatorname{const}}  
\ndef{\conv}{\operatorname{conv}}    
\rndef{\det}{\operatorname{det}}     

\ndef{\detFK}[1]{\Delta\brs{#1}} 
\ndef{\detFKrel}[2]{\Delta_{#2}\brs{#1}} 

\ndef{\adj}{\operatorname{adj}}    
\ndef{\diag}{\operatorname{diag}}    
\ndef{\dist}{\operatorname{dist}}    
\ndef{\dom}{\operatorname{dom}}      
\ndef{\ec}{\operatorname{ec}}        
\ndef{\id}{1}                        
\ndef{\ind}{\operatorname{ind}}      
\ndef{\mydeg}{\operatorname{deg}}    
\ndef{\op}{\operatorname{op}}
\ndef{\rank}{\operatorname{rank}}
\ndef{\res}{\operatorname{res}}      
\ndef{\rng}{\operatorname{ran}}      
\ndef{\sflow}{\operatorname{sf}}     
\ndef{\isf}{\operatorname{isf}}      
\ndef{\sign}{\operatorname{sign}}    
\ndef{\sgn}{\operatorname{sgn}}      
\ndef{\sing}{\operatorname{sing}}    
\ndef{\supp}{\operatorname{supp}}    
\ndef{\tr}{\operatorname{tr}}        
\ndef{\var}{\operatorname{var}}      
\ndef{\vol}{\operatorname{vol}}      
\ndef{\wn}{\operatorname{wn}}        
\ndef{\wres}{\operatorname{wres}}    
\rndef{\Im}{\operatorname{Im}}       
\rndef{\Re}{\operatorname{Re}}       

\ndef{\prng}[1]{\mathrm R_{#1}} 
\ndef{\pker}[1]{\mathrm N_{#1}} 
\ndef{\rprng}[2]{\mathrm R_{#1}^{#2}}           
\ndef{\rpker}[2]{\mathrm N_{#1}^{#2}}           
\ndef{\rsupp}[1]{\supp_r(#1)}
\ndef{\lsupp}[1]{\supp_l(#1)}
\ndef{\rslv}[1]{R_z(#1)}      
\ndef{\HH}{H}                 
\ndef{\tHH}{\tilde \HH}       
\ndef{\VV}{V}                 
\ndef{\Rz}{R_z}               
\ndef{\tRz}{\tR_z}            
\ndef{\psif}[1]{#1^{[1]}} 
\ndef{\WPlus}[1]{W_{#1}(\mbR)} 

\newcommand{\xia}{\xi^{(a)}}
\newcommand{\xis}{\xi^{(s)}}
\newcommand{\mua}{\mu^{(a)}}
\newcommand{\mus}{\mu^{(s)}}

\ndef{\bndl}{\xi}                         
\ndef{\bndlA}{\eta}                       
\ndef{\GlueMap}{\varphi}                  
\ndef{\ChartMap}{h}                       
\ndef{\chern}{\ensuremath{\mathrm{ch}}}
\ndef{\hilb}{\clH}                     
\ndef{\hilba}{\clH^{(a)}}                    
\ndef{\hilbs}{\clH^{(s)}}                    
   \ndef{\hilbasargument}{(\hilb)} 
\ndef{\LpH}[1]{\clL_{#1}\hilbasargument}       
\ndef{\saLpH}[1]{\clL_{sa}^{#1}\hilbasargument}       
\ndef{\clBH}{\clB\hilbasargument}              
\ndef{\ubBH}{\clB_1\hilbasargument}            
\ndef{\clCH}{\clC\hilbasargument}              
\ndef{\clKH}{\clK\hilbasargument}              
\ndef{\clFH}{\clF\hilbasargument}              
\ndef{\clUH}{\clU\hilbasargument}              
\ndef{\clCFH}{{\clC\clF}\hilbasargument}       
\ndef{\saBH}{\clB_{sa}\hilbasargument}         
\ndef{\saCH}{\clC_{sa}\hilbasargument}         
\ndef{\saFH}{\clF_{sa}\hilbasargument}         
\ndef{\saKH}{\clK_{sa}\hilbasargument}         
\ndef{\saCFH}{\clC\clF_{sa}\hilbasargument}    
\ndef{\clUFH}{\clU\clF\hilbasargument}         
\ndef{\Uinj}{\clU_{inj}\hilbasargument}        
\ndef{\UFinj}{\clU\clF_{inj}\hilbasargument}   

\ndef{\spproj}[2]{E^{#1}_{#2}}                      
\ndef{\spprojb}[2]{E^{#2}_{#1}}                     

\ndef{\LpN}[1]{\clL^{#1}(\clN,\tau)}     
\ndef{\saLpN}[1]{\clL^{#1}_{sa}(\clN,\tau)} 
\ndef{\rLpN}[1]{L^{#1}(\clN,\tau)}       
\ndef{\clAND}{(\clA,\clN,D)}             
\ndef{\clBA}{{\clB(\clA)}}
\ndef{\saKN}{{\clK_{sa}(\clN,\tau)}}          
\ndef{\clKN}{{\clK(\clN,\tau)}}          
\ndef{\clKtN}{{\clK(\tilde\clN,\tau)}}   
\ndef{\clFN}{{\clF(\clN,\tau)}}          
\ndef{\saFN}{{\clF_{sa}(\clN,\tau)}}     
\ndef{\clPN}{\clP(\clN)}                 
\ndef{\clQN}{\clQ(\clN,\tau)}            
\ndef{\infPN}{{\clP_\tau^\infty(\clN)}}  
\ndef{\clOF}[2]{\clF_{#1\mbox{-}#2}(\clN,\tau)}         
\ndef{\oind}[2]{{\rm \tau\mbox{-}ind}_{#1\mbox{-}#2}}   
\ndef{\tind}{\tau\mbox{-}\ind}                  
\ndef{\DInd}{\ind_{\clD,\tau}}           
\ndef{\BF}{Breuer-Fredholm}              
\ndef{\skewfred}[2]{$(#1\cdot #2)$ $\tau$\tire Fredholm}   
\ndef{\affl}{\eta}                       
\ndef{\vNa}{von Neumann algebra}         
\ndef{\nsf}{faithful normal semifinite } 
\ndef{\taubrs}[1]{\tau\brackets{#1}}     
\ndef{\sqbrs}[1]{[#1]}        
\ndef{\Sqbrs}[1]{\big[#1\big]}        
\ndef{\SqBrs}[1]{\Big[#1\Big]}        

\ndef{\domd}{\bigcap\limits_{n\ge 0} \dom\;\delta^n}         
\ndef{\DiffOP}{{\rm \clD}}
\ndef{\ADA}{\clA \cup [\clD,\clA]}
\ndef{\DixIdeal}[1]{\LpH{#1,\infty}}               
\ndef{\dixideal}{\ell^{1,\infty}}                  
\ndef{\WDixIdeal}{\LpH{1,\mathrm w}}               
\ndef{\DixIdealPos}[1]{\DixIdeal{#1}_+}            
\ndef{\DixIdealN}[1]{\LpN{#1,\infty}}              
\ndef{\DixIdealNPar}[2]{\clL^{#1,\infty}_{#2}(\clN,\tau)}    
\ndef{\DixIdealNPos}[1]{\LpN{#1,\infty}_+}                   
\ndef{\TrD}{\Tr_\omega}                                      
\ndef{\tauD}{{\tau_\omega}}                                  
\ndef{\ILogN}{\frac 1{\log(1+N)}}
\ndef{\DixNorm}[1]{\norm{#1}_{(1,\infty)}}                   
\ndef{\DixInt}[1]{\ints 0^t \mu_s(#1)\,ds}
\ndef{\DixIntL}[1]{\ints 0^{\lambda_{1/t}(#1)}\mu_s(#1)\,ds}
    \ndef{\SmallIdeal}{{\clL^{1, \mathrm w}}}
    \ndef{\SmallIdealMeas}{{\clL^{1, \mathrm w}_m}}
    \ndef{\DixIntII}[1]{\int_0^t \mu_s(#1)\,ds}
    \ndef{\DixIntf}[1]{\Phi_t(#1)}
    \ndef{\DixIntg}[1]{\Psi_t(#1)}

\ndef{\lpi}{\clL^{1,\pi}(\clN,\tau)}

\ndef{\strl}[1]{\stackrel \longrightarrow {#1}}
\ndef{\IIinfty}{$\mathrm{II}_\infty$\ }

\ndef{\fourier}[1]{\clF(#1)}          
\ndef{\HaarMeasBohrs}{\nu}            
\ndef{\BrownsMeas}{\mu}               
\ndef{\BohrCont}[1]{\tilde{#1}}       
\ndef{\APMean}{{M}}                   
\ndef{\CDSS}{{\clA_B}}                
\ndef{\matr}{{\rm Mat}}               
\ndef{\seque}[1]{\ensuremath{\{#1_n\}_{n=1}^\infty}}    
\ndef{\sequen}[2]{\ensuremath{\{#1_#2\}_{#2=1}^\infty}}    
\ndef{\Seque}[1]{\ensuremath{\left(#1_0,#1_1,#1_2,\dots\right)}}    
\ndef{\Cesaro}{H}                           
\ndef{\CesaroRPlus}{M}                      
\ndef{\Dilation}{D}                         
\ndef{\Shift}{T}                            

\ndef{\norm}[1]{\left\Vert#1\right\Vert}    
\ndef{\TrNorm}[1]{\norm{#1}_1}              
\ndef{\HSNorm}[1]{\norm{#1}_2}              
\ndef{\InftyNorm}[1]{\norm{#1}_\infty}      
\ndef{\normQN}[1]{\norm{#1}_{\clQN}}        
\ndef{\clLpnorm}[2]{\norm{#2}_{\clL^{#1}}}    
\ndef{\clLnorm}[1]{\clLpnorm{1}{#1}}    

\ndef{\ccurve}{\gamma}                      

\ndef{\abs}[1]{\left\lvert#1\right\rvert}   
\ndef{\set}[1]{\left\{#1\right\}}           
\ndef{\brackets}[1]{\left(#1\right)}        
\ndef{\brs}[1]{\brackets{#1}}               
\ndef{\Brs}[1]{\big(#1\big)}                
\ndef{\BRS}[1]{\Big(#1\Big)}                
\ndef{\scal}[2]{\left\la #1,#2\right\ra}               
\ndef{\precprec}{\prec\!\!\!\prec}
\ndef{\qeq}{\stackrel?=}
\ndef{\spectrum}[1]{\sigma_{#1}} 
\ndef{\spectruma}[1]{\sigma^{(a)}_{#1}} 
\ndef{\numrange}[1]{\mathrm{W}(#1)}                         
\rndef{\emptyset}{\varnothing}                              
\ndef{\csupp}{c}                           
\ndef{\closure}[1]{\overline{#1}}
\ndef{\linspan}[1]{\mathrm{span}\ {#1}}
\ndef{\bddborel}[1]{B(#1)}                 
\ndef{\charfunc}{\chi}
\ndef{\FrDer}{\euD}                        
\ndef{\LieDer}[1]{\pounds_{#1}\,}          
\ndef{\dds}{\left.\frac d{ds} \right|_{s = 0}}
\ndef{\ortcmp}[1]{#1^{\scriptscriptstyle \perp}}            
\ndef{\Laplace}{\Delta}                    

\ndef{\matrPQ}[3]
{
    \left(
      \begin{array}{cc}
        #1_{11} & #1_{12} \\
        #1_{21} & #1_{22}
      \end{array}
    \right)_{[#2,#3]}
}

\ndef{\margOK}{\marginpar{\bf \small OK}}

\newcounter{margcomcount}
\ndef{\margcom}[1]{\marginpar{\bf \small #1} \addtocounter{margcomcount}{1}
   \index{\indexcom{{\bf COMMENT: #1}}}}

\newcounter{margproof}
\ndef{\margproof}{\marginpar{\bf \small PROOF} \addtocounter{margproof}{1}
  \index{**** \indexcom{{\bf PROOF}}}}

\newcounter{margdetails}
\ndef{\margdetails}{\marginpar{\bf Details} \addtocounter{margdetails}{1}
  \index{**** \indexcom{{\bf DETAILS}}}}

\newcounter{margproofb}
\ndef{\margproofb}[1]{\marginpar{\bf \small Proof(B) #1} \addtocounter{margproofb}{1}
  \index{**** \indexcom{{\bf PROOF(B): #1}}}}

\newcounter{margdetailsb}
\ndef{\margdetailsb}[1]{\marginpar{\bf \small Details(B)} \addtocounter{margdetailsb}{1}
  \index{**** \indexcom{{\bf DETAILS(B): \\ #1}}}}

\newcounter{margdetailsc}
\ndef{\margdetailsc}[1]{\marginpar{\bf \small Details(C)} \addtocounter{margdetailsc}{1}
  \index{**** \indexcom{{\bf DETAILS(C): \\ #1}}}}

\newcounter{margcomcountb}
\ndef{\margcomb}[1]{\marginpar{\bf \small #1} \addtocounter{margcomcountb}{1}
   \index{\indexcom{{\bf COMMENT(B): \\ #1}}}}

  \rndef{\margcom}[1]{}
  \rndef{\margproof}{}
  \rndef{\margdetails}{}
  \rndef{\margproofb}[1]{}
  \rndef{\margdetailsb}[1]{}
  \rndef{\margdetailsc}[1]{}
  \rndef{\margcomb}[1]{}

\ndef{\mytimes}{\!\times\!}
\ndef{\sss}[1]{\subsubsection{}\label{#1}}

\rndef{\phi}{\varphi}
\ndef{\OpenUnitDisk}{D}
\ndef{\RHS}{RHS}                            
\ndef{\LHS}{LHS} 
\ndef{\ttt}{\Leftrightarrow}
\ndef{\then}{\Rightarrow}
\ndef{\tto}{\longrightarrow}
\ndef{\nno}{\nonumber\\}
\ndef{\newn}[1]{\index{#1} {\bfseries #1}}       
\ndef{\la}{\langle}
\ndef{\ra}{\rangle} \ndef{\dbar}{{\;\bar{\phantom{o}} \!\!\!\! d}}
\ndef{\stl}[1]{\stackrel{\vbox to 0pt{\vss\hbox{$\scriptstyle #1$}}}}
\ndef{\mathcomment}[1]{{\hfill \qquad\qquad\qquad\qquad\qquad\text{by (#1)}}}        
\ndef{\mathcomm}[1]{{\hfill \qquad\qquad\qquad\qquad\qquad\text{#1}}}        
\ndef{\details}[1]{\smallskip\begin{center} {\bf Here:}
#1\end{center}\medskip} \ndef{\indexcom}[1]{ --- #1}
\ndef{\longsim}{\ \sim \ }              
\ndef{\tire}{-}              
\ndef{\intinfinf}{\int_{-\infty}^\infty}
\ndef{\refnsftrace}{\cite[V.\,2.\,1]{TakI}} 
\ndef{\refaffloper}{\cite[IV.\,5, Exercise 3]{TakI}} 
\ndef{\refsemifinvNa}{\cite[V.\,1.\,21]{TakI}} 
\ndef{\reftaumeasurable}{\cite[Definition 1.2]{FK86PJM}} 
\ndef{\reftautraceclassaffl}{\cite[V.2, p.\,320]{TakI}} 
\ndef{\refinvoperideal}{\cite[Appendix A.2]{CP2}} 
\ndef{\reftautracenorm}{\cite[V.2, p.\,320]{TakI}} 
\ndef{\reftaucompact}{\cite{}} 
\ndef{\reftauFredholm}{\cite[Appendix B]{PR94JFA}} 

     \ndef{\npartial}{\slash\!\!\!\partial}
     \ndef{\Heis}{\operatorname{Heis}}
     \ndef{\Solv}{\operatorname{Solv}}
     \ndef{\Spin}{\operatorname{Spin}}
     \ndef{\SO}{\operatorname{SO}}
     \ndef{\Index}{\operatorname{index}}

             \ndef{\p}{\partial}
             \ndef{\dd}{|\clD|}
             \ndef{\n}{\parallel}

     \ndef{\gf}[2]{\genfrac{}{}{0pt}{}{#1}{#2}}
     \ndef{\ta}{\widetilde{\alpha}}
     \ndef{\tb}{\widetilde{\beta}}
     \ndef{\txi}{\widetilde{\xi}}
     \ndef{\tk}{\widetilde{K}}
     \ndef{\CGh}{\widetilde{\CG}}
     \ndef{\boe}{{\bf e}}\ndef{\bt}{{\bf t}}
     \ndef{\vth}{\vartheta}
     \ndef{\db}{\overline{\partial}}
     \ndef{\hV}{\hat{V}}
     \ndef{\cag}{{\clA^\Gamma}}
     \ndef{\sind}{\sigma{\rm -ind}}

\newcounter{slidecount}
\catcode`@=11
\newcommand{\slide}[2]{
  \newpage
  \addtocounter{slidecount}{1}
  \renewcommand{\@oddhead}{{\small Advanced Mathematics 1A -- 2009 \hfil Slide #1-\arabic{slidecount}}}
  \begin{center} \bf #2 \end{center}
}
\catcode`@=12

%% file: MyMakeLit.tex
\let\LatexCite=\cite  

\let\ifnumref\iffalse 

\ndef{\ifuncited}[4]{\expandafter\ifx\csname used#4\endcsname\relax}

\ndef{\ifcited}[4]{\expandafter\ifx\csname used#4\endcsname\relax\else}

  \ndef{\papertitle}[1]{ \emph{#1}, }
  \ndef{\paperauthor}[2]{#2}  
  \ndef{\pbbi}[9]{%
      \ifcited{#1}{#2}{#3}{#5}%
        \ifnumref%
          \bibitem{#5}\paperauthor{#1}{#6},\papertitle{#7}#8.%
        \else%
          \advance #9 by 1%
          \ifnum#9<1%
            \bibitem[#4]{#5}\paperauthor{#1}{#6}, \papertitle{#7}#8.%
          \else%
            \bibitem[#4$_{\the#9}$]{#5}\paperauthor{#1}{#6},\papertitle{#7}#8.%
          \fi%
        \fi%
      \fi%
  }
  \ndef{\mbbi}[8]{%
     \ifcited{#1}{#2}{#3}{#5}%
        \ifnumref%
          \bibitem{#5}\paperauthor{#1}{#6},\papertitle{#7}#8.%
        \else%
          \bibitem[#4]{#5}\paperauthor{#1}{#6},\papertitle{#7}#8.%
        \fi%
     \fi%
  }

\ndef{\AddCite}[1]{%
   \ifuncited{0}{0}{0}{#1}%
     \expandafter\gdef\csname used#1\endcsname {}%
   \fi%
}

\def\ProcessCite#1,{%
     \ifx\relax#1%
         \let\next=\relax%
     \else%
         \AddCite{#1}%
         \let\next=\ProcessCite%
     \fi%
     \next%
}

\ndef{\AddCites}[1]{\ProcessCite#1,\relax,}

\ndef{\CiteWithoutExtension}[1]{%
   \AddCites{#1}%
   \LatexCite{#1}%
}

\def\CiteWithExtension[#1]#2{%
   \AddCites{#2}%
   \LatexCite[#1]{#2}%
}

\ndef{\CleverCite}{%
    \ifx\NChar[ %
       \let\MyCite=\CiteWithExtension %
    \else %
       \let\MyCite=\CiteWithoutExtension %
    \fi %
    \MyCite%
}

\renewcommand{\cite}{\futurelet\NChar\CleverCite}

      \ndef{\volume}[1]{{\bf #1}}
      \ndef{\VolYearPP}[3]{\ifnum#2=0 (to appear)\else\volume{#1} (#2), #3\fi}
      \ndef{\VolNoYearPP}[4]{\ifnum#3=0 (to appear)\else\volume{#1} #2 (#3), #4\fi}
      \ndef{\libcode}[1]{}

\ndef{\jnActaMath}[3]{Acta Math. \VolYearPP{#1}{#2}{#3}}                       
\ndef{\jnAdvMath}[3]{Adv. in~Math. \VolYearPP{#1}{#2}{#3}}                     
\ndef{\jnAlgAnal}[3]{Algebra i~Analiz \VolYearPP{#1}{#2}{#3}}
\ndef{\jnAmerJMath}[3]{Amer. J. Math. \VolYearPP{#1}{#2}{#3}}                  
\ndef{\jnAmerMathMonth}[3]{Amer. Math. Monthly \VolYearPP{#1}{#2}{#3}}         
\ndef{\jnAnnMath}[4]{Ann. of~Math. \VolNoYearPP{#1}{#2}{#3}{#4}}               
\ndef{\jnAnalMath}[3]{J. Anal. Math. \VolYearPP{#1}{#2}{#3}}                   
\ndef{\jnBullLondMathSoc}[3]{Bull. London Math. Soc. \VolYearPP{#1}{#2}{#3}}   
\ndef{\jnBullAMS}[3]{Bull. Amer. Math. Soc. \VolYearPP{#1}{#2}{#3}}   
\ndef{\jnCanMathBull}[3]{Canad. Math. Bull. \VolYearPP{#1}{#2}{#3}}            
\ndef{\jnCanMath}[3]{Canad. J.~Math. \VolYearPP{#1}{#2}{#3}}             
\ndef{\jnCommMathPhys}[3]{Comm. Math. Phys. \VolYearPP{#1}{#2}{#3}}             
\ndef{\jnCommPDE}[3]{Comm. Partial Differential Equations \VolYearPP{#1}{#2}{#3}}             
\ndef{\jnComptRendue}[3]{C.\,R.~Acad. Sci. Paris S\'er. A-B \VolYearPP{#1}{#2}{#3}}      
\ndef{\jnContMath}[3]{Contemporary Math. \VolYearPP{#1}{#2}{#3}}               %
\ndef{\jnDukeMJ}[3]{Duke Math. J. \VolYearPP{#1}{#2}{#3}}
\ndef{\jnDiffGeom}[3]{J.~Diff. Geom. \VolYearPP{#1}{#2}{#3}}                   
\ndef{\jnErgodicTheory}[3]{Ergodic Theory and Dynamical Systems \VolYearPP{#1}{#2}{#3}} 
\ndef{\jnFuncAnal}[3]{J.~Functional Analysis \VolYearPP{#1}{#2}{#3}}           
\ndef{\jnFunkAnalPril}[4]{Funct. Anal. Appl. \VolNoYearPP{#1}{#2}{#3}{#4}}  
\ndef{\jnGAFA}[3]{GAFA \VolYearPP{#1}{#2}{#3}}                                 
\ndef{\jnIHES}[3]{IHES Publ. Math. (Paris) \VolYearPP{#1}{#2}{#3}}             
\ndef{\jnIEOT}[3]{Integral Equations Operator Theory   \VolYearPP{#1}{#2}{#3}} 
\ndef{\jnIsrMath}[3]{Israel J.~Math. \VolYearPP{#1}{#2}{#3}}                   
\ndef{\jnKTheory}[3]{K-Theory \VolYearPP{#1}{#2}{#3}}                          
\ndef{\jnLetMathPhys}[3]{Lett. Math. Phys. \VolYearPP{#1}{#2}{#3}}             
\ndef{\jnMathAnn}[3]{Math. Ann. \VolYearPP{#1}{#2}{#3}}                        
\ndef{\jnMathAnalAppl}[3]{J.~Math. Anal. and Appl. \VolYearPP{#1}{#2}{#3}}     
\ndef{\jnMathNachr}[3]{Math. Nachr. \VolYearPP{#1}{#2}{#3}}
\ndef{\jnMathPhys}[3]{J. Math. Phys. \VolYearPP{#1}{#2}{#3}}
\ndef{\jnMathSocJap}[3]{J. Math. Soc. Japan \VolYearPP{#1}{#2}{#3}}
\ndef{\jnOperTheory}[3]{J.~Operator Theory \VolYearPP{#1}{#2}{#3}}             
\ndef{\jnPacJMath}[3]{Pacific J.~Math. \VolYearPP{#1}{#2}{#3}}                  
\ndef{\jnPositivity}[3]{Positivity \VolYearPP{#1}{#2}{#3}}
\ndef{\jnProcAmerMS}[3]{Proc. Amer. Math. Soc. \VolYearPP{#1}{#2}{#3}}         
\ndef{\jnProcCambPhilSoc}[3]{Math. Proc. Camb. Phil. Soc. \VolYearPP{#1}{#2}{#3}}
\ndef{\jnReineAngew}[3]{J.~Reine Angew. Math. \VolYearPP{#1}{#2}{#3}}          
\ndef{\jnTokyoMath}[3]{Tokyo J.~Math. \VolYearPP{#1}{#2}{#3}}
\ndef{\jnTopology}[3]{Topology \VolYearPP{#1}{#2}{#3}}
\ndef{\jnTransAmerMathSoc}[3]{Trans. Amer. Math. Soc. \VolYearPP{#1}{#2}{#3}}
\ndef{\jnIzvANSSSR}[3]{Izv. Akad. Nauk SSSR, Ser. Mat. \VolYearPP{#1}{#2}{#3}}
\ndef{\jnIzvVyshUchZav}[3]{Izv. Vyssh. Uch. Zav., Mat. \VolYearPP{#1}{#2}{#3} (Russian)}
\ndef{\jnIzdatLenUniv}[2]{Izdat. Leningrad. Univ., Leningrad, (#1), #2 (Russian)}
\ndef{\jnFieldsInsComm}[3]{Fields Inst. Comm. \VolYearPP{#1}{#2}{#3}}
\ndef{\jnDoklANSSSR}[3]{Dokl. Akad. Nauk SSSR \VolYearPP{#1}{#2}{#3}}
\ndef{\jnMatZametki}[3]{Matem. zametki \VolYearPP{#1}{#2}{#3}}
\ndef{\jnRussMathSurvey}[3]{Russian Math. Surveys \VolYearPP{#1}{#2}{#3}}
\ndef{\jnSibMathJ}[3]{Sib. Math.~J. \VolYearPP{#1}{#2}{#3}}
\ndef{\jnSovMath}[3]{J.~Soviet math. \VolYearPP{#1}{#2}{#3}}
\ndef{\jnTransMoscMathSoc}[3]{Trans. Moscow Math. Soc. \VolYearPP{#1}{#2}{#3}}
\ndef{\jnUMN}[3]{Uspekhi Mat. Nauk \VolYearPP{#1}{#2}{#3}}

\ndef{\bkTransMathMon}[2]{Trans. Math. Monographs, AMS, \volume{#1}, #2}

\ndef{\pbBirkhauser}[1]{Birkh\"auser, Boston, #1}
\ndef{\pbFactorial}[1]{Moscow, Factorial, #1}
\ndef{\pbGauthier}[1]{Gauthier-Villars, Paris, #1}
\ndef{\pbNauka}[1]{Moscow, Nauka, #1 (Russian)}
\ndef{\pbNaukaR}[1]{Москва, Наука, #1}
\ndef{\pbPrinceton}[1]{Princeton University Press, Princeton, New Jersey, #1}
\ndef{\pbPublPerish}[1]{Publish or Perish Inc., Berkeley, #1}
\ndef{\pbSpringer}[1]{Springer-Verlag, #1}

\ndef{\myauthor}[1]{\mbox{#1}}

\ndef{\Agmon}{\myauthor{Sh.\,Agmon}}
\ndef{\Ahiezer}{\myauthor{N.\,I.\,Ahiezer}}
\ndef{\Arazy}{\myauthor{J.\,Arazy}}
\ndef{\Aronszajn}{\myauthor{N.\,Aronszajn}}
\ndef{\Astashkin}{\myauthor{S.\,V.\,Astashkin}}
\ndef{\Atiyah}{\myauthor{M.\,Atiyah}}
\ndef{\Avron}{\myauthor{J.\,E.\,Avron}}
\ndef{\Azamov}{\myauthor{N.\,A.\,Azamov}}
\ndef{\Banach}{\myauthor{S.\,Banach}}
\ndef{\Benameur}{\myauthor{M-T.\,Benameur}}
\ndef{\Bennett}{\myauthor{C.\,Bennett}}
\ndef{\Berezin}{\myauthor{F.\,A.\,Berezin}}
\ndef{\Berline}{\myauthor{N.\,Berline}}
\ndef{\Birman}{\myauthor{M.\,Sh.\,Birman}}
\ndef{\Blackadar}{\myauthor{B.\,Blackadar}}
\ndef{\Bogolyubov}{\myauthor{N.\,N.\,Bogolyubov}}
\ndef{\Bonsall}{\myauthor{F.\,F.\,Bonsall}}
\ndef{\Bony}{\myauthor{J.\,F.\,Bony}}
\ndef{\BoosBavnbek}{\myauthor{B.\,Boo$\beta$-Bavnbek}}
\ndef{\Bott}{\myauthor{R.\,Bott}}
\ndef{\Branges}{\myauthor{L.\,de Branges}}
\ndef{\Bratteli}{\myauthor{O.\,Bratteli}}
\ndef{\Bredon}{\myauthor{G.\,E.\,Bredon}}
\ndef{\Breuer}{\myauthor{M.\,Breuer}}
\ndef{\Brown}{\myauthor{L.\,G.\,Brown}}
\ndef{\Bruneau}{\myauthor{V.\,Bruneau}}
\ndef{\Buslaev}{\myauthor{V.\,S.\,Buslaev}}
\ndef{\Carey}{\myauthor{A.\,L.\,Carey}}
\ndef{\CareyRW}{\myauthor{R.\,W.\,Carey}} 
\ndef{\Cartan}{\myauthor{H.\,Cartan}}
\ndef{\Chilin}{\myauthor{V.\,I.\,Chilin}}
\ndef{\Coburn}{\myauthor{L.\,A.\,Coburn}}
\ndef{\Connes}{\myauthor{A.\,Connes}}
\ndef{\Cornfeld}{\myauthor{I.\,P.\,Cornfeld}}
\ndef{\Daletskii}{\myauthor{Yu.\,L.\,Daletski\u\i}}   
\ndef{\Dixmier}{\myauthor{J.\,Dixmier}}
\ndef{\DoddsPG}{\myauthor{P.\,G.\,Dodds}}
\ndef{\DoddsTK}{\myauthor{T.\,K.\,Dodds}}
\ndef{\Douglas}{\myauthor{R.\,G.\,Douglas}}
\ndef{\Dubrovin}{\myauthor{B.\,A.\,Dubrovin}}
\ndef{\Dugundji}{\myauthor{J.\,Dugundji}}
\ndef{\Duncan}{\myauthor{J.\,Duncan}}
\ndef{\Dunford}{\myauthor{N.\,Dunford}}
\ndef{\Dykema}{\myauthor{K.\,J.\,Dykema}}
\ndef{\Edwards}{\myauthor{R.\,E.\,Edwards}}
\ndef{\Eilenberg}{\myauthor{S.\,Eilenberg}}
\ndef{\Entina}{\myauthor{S.\,B.\,\`Entina}}
\ndef{\Fack}{\myauthor{T.\,Fack}} 
\ndef{\Faddeev}{\myauthor{L.\,D.\,Faddeev}}
\ndef{\Farber}{\myauthor{M.\,Farber}}
\ndef{\Farforovskaya}{\myauthor{Yu.\,B.\,Farforovskaya}}
\ndef{\Federer}{\myauthor{H.\,Federer}}
\ndef{\Fedosov}{\myauthor{B.\,V.\,Fedosov}}
\ndef{\Figiel}{\myauthor{T.\,Figiel}} 
\ndef{\Figueroa}{\myauthor{H.\,Figueroa}}
\ndef{\Fillmore}{\myauthor{P.\,A.\,Fillmore}}
\ndef{\Fomenko}{\myauthor{A.\,T.\,Fomenko}} 
\ndef{\Fomin}{\myauthor{S.\,V.\,Fomin}}
\ndef{\Frohlich}{\myauthor{J.\,Fr\"ohlich}}
\ndef{\Fuglede}{\myauthor{B.\,Fuglede}}
\ndef{\Furutani}{\myauthor{K.\,Furutani}}
\ndef{\Gelfand}{\myauthor{I.\,M.\,Gelfand}}
\ndef{\Gesztesy}{\myauthor{F.\,Gesztesy}}     
\ndef{\Getzler}{\myauthor{E.\,Getzler}} 
\ndef{\Gilkey}{\myauthor{P.\,B.\,Gilkey}}
\ndef{\Gitler}{\myauthor{S.\,Gitler}}
\ndef{\Glazman}{\myauthor{I.\,M.\,Glazman}}
\ndef{\Glimm}{\myauthor{J.\,Glimm}}
\ndef{\Gohberg}{\myauthor{I.\,C.\,Gohberg}}
\ndef{\Goldshtein}{\myauthor{Ya.\,Goldshtein}}
\ndef{\Golze}{\myauthor{F.\,Golze}}
\ndef{\GraciaBondia}{\myauthor{J.\,M.\,Gracia-Bond\'{i}a}}
\ndef{\Greenleaf}{\myauthor{F.\,P.\,Greenleaf}}
\ndef{\Gromov}{\myauthor{M.\,Gromov}}
\ndef{\Gunning}{\myauthor{R.\,C.\,Gunning}}
\ndef{\Haagerup}{\myauthor{U.\,Haagerup}}
\ndef{\Haag}{\myauthor{R.\,Haag}}
\ndef{\Halmos}{\myauthor{Halmos}}
\ndef{\Hardy}{\myauthor{G.\,H.\,Hardy}}
\ndef{\Herbst}{\myauthor{I.\,W.\,Herbst}}
\ndef{\Higson}{\myauthor{N.\,Higson}}  
\ndef{\Hoermander}{\myauthor{L.\,Hoermander}} 
\ndef{\Hoffman}{\myauthor{K.\,Hoffman}} 
\ndef{\Ito}{\myauthor{K.\,Ito}}
\ndef{\Jaffe}{\myauthor{A.\,Jaffe}}
\ndef{\James}{\myauthor{I.\,M.\,James}}
\ndef{\Javrjan}{\myauthor{V.\,A.\,Javrjan}}
\ndef{\Jitomirskaya}{\myauthor{S.\,Jitomirskaya}}
\ndef{\Kadison}{\myauthor{R.\,V.\,Kadison}}
\ndef{\Kalton}{\myauthor{N.\,J.\,Kalton}} 
\ndef{\Kato}{\myauthor{T.\,Kato}} 
\ndef{\Kobayashi}{\myauthor{S.\,Kobayashi}}
\ndef{\Koplienko}{\myauthor{L.\,S.\,Koplienko}}
\ndef{\Korotyaev}{\myauthor{E.\,Korotyaev}}
\ndef{\Kosaki}{\myauthor{H.\,Kosaki}}
\ndef{\Kostrykin}{\myauthor{V.\,Kostrykin}}
\ndef{\Kotani}{\myauthor{S.\,Kotani}}
\ndef{\Krein}{\myauthor{Kre\u\i n}}
\ndef{\KreinMG}{\myauthor{M.\,G.\,Kre\u\i n}}
\ndef{\KreinSG}{\myauthor{S.\,G.\,Kre\u\i n}}
\ndef{\Kuroda}{\myauthor{S.\,T.\,Kuroda}}
\ndef{\Leichtnam}{\myauthor{E.\,Leichtnam}}
\ndef{\Lesch}{\myauthor{M.\,Lesch}}
\ndef{\Lesniewski}{\myauthor{A.\,Lesniewski}}
\ndef{\Levitan}{\myauthor{B.\,M.\,Levitan}}
\ndef{\Lidskii}{\myauthor{V.\,B.\,Lidskii}}
\ndef{\Lifshits}{\myauthor{I.\,M.\,Lifshits}}
\ndef{\Lindenstrauss}{\myauthor{J.\,Lindenstrauss}}
\ndef{\Loday}{\myauthor{J.-L.\,Loday}}
\ndef{\Lord}{\myauthor{S.\,Lord}}      
\ndef{\Lorentz}{\myauthor{G.\,Lorentz}}
\ndef{\Magnus}{\myauthor{W.\,Magnus}}
\ndef{\Makarov}{\myauthor{K.\,A.\,Makarov}}
\ndef{\MakarovN}{\myauthor{N.\,Makarov}}
\ndef{\Mathai}{\myauthor{V.\,Mathai}}         
\ndef{\McKean}{\myauthor{H.\,P.\,McKean}}
\ndef{\Mishchenko}{\myauthor{A.\,S.\,Mishchenko}}
\ndef{\Molchanov}{\myauthor{S.\,A.\,Molchanov}}
\ndef{\Moore}{\myauthor{C.\,C.\,Moore}}
\ndef{\Moscovici}{\myauthor{H.\,Moscovici}}  
\ndef{\Motovilov}{\myauthor{A.\,K.\,Motovilov}}
\ndef{\Moyer}{\myauthor{R.\,D.\,Moyer}}
\ndef{\Naboko}{\myauthor{S.\,N.\,Naboko}}
\ndef{\Narasimhan}{\myauthor{R.\,Narasimhan}}
\ndef{\Nomizu}{\myauthor{K.\,Nomizu}}
\ndef{\Novikov}{\myauthor{S.\,P.\,Novikov}}
\ndef{\Osterwalder}{\myauthor{K.\,Osterwalder}}
\ndef{\Patodi}{\myauthor{V.\,Patodi}}
\ndef{\Pagter}{\myauthor{B.\,de~Pagter}}  
\ndef{\Pastur}{\myauthor{L.\,A.\,Pastur}}  
\ndef{\Pavlov}{\myauthor{B.\,S.\,Pavlov}}
\ndef{\Pedersen}{\myauthor{G.\,K.\,Pedersen}}
\ndef{\Peller}{\myauthor{V.\,V.\,Peller}}
\ndef{\Perera}{\myauthor{V.\,S.\,Perera}}
\ndef{\Petunin}{\myauthor{Ju.\,I.\,Petunin}}
\ndef{\Phillips}{\myauthor{J.\,Phillips}}  
\ndef{\Piazza}{\myauthor{P.\,Piazza}}   
\ndef{\Pincus}{\myauthor{J.\,D.\,Pincus}}   
\ndef{\Poincare}{Poincar\'e}
\ndef{\Postnikov}{\myauthor{M.\,M.\,Postnikov}} 
\ndef{\Prinzis}{\myauthor{R.\,Prinzis}}
\ndef{\Privalov}{\myauthor{I.\,I.\,Privalov}}
\ndef{\Pushnitski}{\myauthor{A.\,B.\,Pushnitski}} 
\ndef{\Raeburn}{\myauthor{I.\,Raeburn}}
\ndef{\Raikov}{\myauthor{G.\,Raikov}}
\ndef{\Reed}{\myauthor{M.\,Reed}}
\ndef{\Rennie}{\myauthor{A.\,Rennie}}
\ndef{\Rickart}{\myauthor{C.\,E.\,Rickart}}
\ndef{\Riesz}{\myauthor{F.\,Riesz}}
\ndef{\Ringrose}{\myauthor{J.\,Ringrose}}
\ndef{\Rio}{\myauthor{R.\,del Rio}}
\ndef{\Robinson}{\myauthor{D.\,Robinson}}
\ndef{\Rossi}{\myauthor{H.\,Rossi}}
\ndef{\Rudin}{\myauthor{W.\,Rudin}}
\ndef{\Ruelle}{\myauthor{D.\,Ruelle}}
\ndef{\Ruzhansky}{\myauthor{M.\,Ruzhansky}}
\ndef{\Sakai}{\myauthor{Sh.\,Sakai}}
\ndef{\Sargsjan}{\myauthor{I.\,S.\,Sargsjan}}
\ndef{\Sato}{\myauthor{H.\,Sato}}
\ndef{\Schaeffer}{\myauthor{D.\,G.\,Schaeffer}}
\ndef{\Schluchtermann}{\myauthor{G.\,Schluchtermann}}
\ndef{\Schochet}{\myauthor{C.\,Schochet}}
\ndef{\SchroedingerE}{\myauthor{E.\,Schr\"odinger}}
\ndef{\Schroedinger}{\myauthor{Schr\"odinger}}
\ndef{\Schrohe}{\myauthor{E.\,Schrohe}}
\ndef{\Schwartz}{\myauthor{J.\,T.\,Schwartz}}
\ndef{\Sedaev}{\myauthor{A.\,A.\,Sedaev}}
\ndef{\Seiler}{\myauthor{R.\,Seiler}}
\ndef{\Semenov}{\myauthor{E.\,M.\,Semenov}}
\ndef{\Shabat}{\myauthor{B.\,V.\,Shabat}}
\ndef{\Shafarevich}{\myauthor{I.\,R.\,Shafarevich}}
\ndef{\Sharpley}{\myauthor{R.\,Sharpley}}
\ndef{\Shilov}{\myauthor{G.\,E.\,Shilov}}
\ndef{\Shirkov}{\myauthor{D.\,V.\,Shirkov}}
\ndef{\Shubin}{\myauthor{M.\,A.\,Shubin}}
\ndef{\Silverman}{\myauthor{H.\,Silverman}}
\ndef{\Simon}{\myauthor{B.\,Simon}}
\ndef{\Sinai}{\myauthor{Ya.\,G.\,Sinai}}
\ndef{\Singer}{\myauthor{I.\,M.\,Singer}}
\ndef{\Solomyak}{\myauthor{M.\,Z.\,Solomyak}}
\ndef{\Soloviev}{\myauthor{Yu.\,P.\,Soloviev}}
\ndef{\Spivak}{\myauthor{M.\,Spivak}}
\ndef{\Stein}{\myauthor{E.\,M.\,Stein}}
\ndef{\Stenkin}{\myauthor{V.\,V.\,Sten'kin}}
\ndef{\Stratila}{\myauthor{S.\,Stratila}}
\ndef{\Sucheston}{\myauthor{L.\,Sucheston}}
\ndef{\Sukochev}{\myauthor{F.\,A.\,Sukochev}}
\ndef{\Switzer}{\myauthor{R.\,M.\,Switzer}}
\ndef{\SzNagy}{\myauthor{B.\,Sz.-Nagy}}
\ndef{\Takesaki}{\myauthor{M.\,Takesaki}}
\ndef{\Taylor}{\myauthor{M.\,E.\,Taylor}}
\ndef{\Treves}{\myauthor{F.\,Treves}}
\ndef{\Troitsky}{\myauthor{E.\,V.\,Troitsky}}
\ndef{\Tzafriri}{\myauthor{L.\,Tzafriri}}
\ndef{\Varilly}{\myauthor{J.\,C.\,V\'{a}rilly}}
\ndef{\Vergne}{\myauthor{M.\,Vergne}}
\ndef{\Vladimirov}{\myauthor{V.\,S.\,Vladimirov}}
\ndef{\Voiculescu}{\myauthor{D.\,Voiculescu}}
\ndef{\Weiss}{\myauthor{G.\,Weiss}}
\ndef{\Wells}{\myauthor{R.\,O.\,Wells}}
\ndef{\Williams}{\myauthor{J.\,P.\,Williams}}
\ndef{\Winkler}{\myauthor{S.\,Winkler}}
\ndef{\Witten}{\myauthor{E.\,Witten}}
\ndef{\Wodzicki}{\myauthor{M.\,Wodzicki}}
\ndef{\Wojciechowski}{\myauthor{K.\,P.\,Wojciechowski}}
\ndef{\Yafaev}{\myauthor{D.\,R.\,Yafaev}}
\ndef{\Yosida}{\myauthor{K.\,Yosida}}
\ndef{\Zsido}{\myauthor{L.\,Zsido}}


%% file: MyListOfRef.tex

\rndef{\emph}[1]{{\it #1}}

\mathsurround 0pt
\ndef{\AndSoOn}{$\dots$}
